\documentclass[english,12pt]{article}
\usepackage[T1]{fontenc}
\usepackage[utf8]{inputenc}
\usepackage{lmodern}
\usepackage[a4paper]{geometry}
\usepackage{babel}
\usepackage{amssymb}
\usepackage{amsmath}
\usepackage{amsthm}
\usepackage{esint}
\usepackage{mathrsfs}
\usepackage{graphicx}
\usepackage{pifont}
\usepackage{amsfonts}
\usepackage{comment}
\usepackage{tikz}

\usepackage{algorithm}
\usepackage{algorithmicx}
\usepackage{algpseudocode}

\usepackage[backend=bibtex,style=numeric]{biblatex}
\usepackage{esint}
\usepackage{dsfont}
\usepackage{subcaption}
\usepackage{wasysym}
\usepackage{bbm}
\usepackage[all]{xy}
\usepackage{hyperref}
\usepackage{comment}
\theoremstyle{theorem}
\newtheorem{theorem}{Theorem}[section]
\newtheorem{lemma}[theorem]{Lemma}
\newtheorem{definition}[theorem]{Definition}
\newtheorem{proposition}[theorem]{Proposition}
\newtheorem{corollaire}[theorem]{Corollary}
\newtheorem{remark}[theorem]{Remark}

\newtheorem{claim}[theorem]{Claim}

\title{Markov properties for the vertical edge profile in random labelled trees}
\author{Alexis Metz-Donnadieu}
\begin{document}
	\maketitle
	\begin{abstract}
		We study a broad class of random labelled trees in which integer-valued labels evolve along the edges according to increments in 
		$\{-1, 0, 1\}$. These models include e.g. branching random walks, embedded complete and incomplete binary trees, random Cayley and plane trees with uniform displacements along edges. Motivated by recent work suggesting a Markovian structure in the vertical profile of such trees, we introduce the vertical edge profile, which counts both oriented edges connecting label $k-1$ to label $k$ and oriented edges connecting label $k$ to label $k-1$.
		
		We show that the vertical edge profile forms a time-homogeneous Markov chain for a wide class of models, and this remains true (provided we enrich this process by the total mass of the tree below each label) if we condition on the total size of the tree. We give explicit transition kernels in the case of labelled incomplete binary trees, which are closely related to known enumeration formulas. To establish these results, we study a decomposition of labelled trees into excursions above and below fixed label levels, yielding a forest structure with tractable probabilistic laws.
		
		We further explain briefly how these findings connect to the theory of super-Brownian motion and the Integrated Super-Brownian Excursion (ISE). In a companion paper, we show that the vertical edge profile converges—after rescaling—to the local time and its derivative of Brownian motion indexed by the Brownian tree.
	\end{abstract}
	
	\section{Introduction}
	\subsection{Presentation of the main results}
	
	In this work, we consider some families of random labelled trees, with integer-valued labels (which may be positive or negative) such that the label increments along edges are constrained to lie in $\{-1, 0, 1\}$. The main class of models we focus on is defined as follows.
 Fix a critical or subcritical distribution $\xi$ on $\mathbb N_{\geq 0}$ and for every integer $d\geq 1$, fix a probability distribution $\eta^{(d)}$ on $\{-1, 0, 1\}^d$. We begin by sampling a Galton--Watson tree $T$ with offspring distribution $\xi$, viewed as a random plane tree so that children of every individual are ordered (note that $T$ is finite since $\xi$ is critical or subcritical). Next, conditionally on $T$ and for every vertex $v\in V(T)$ with at least one children, sample a random vector $[v]_T\in \{-1, 0, 1\}^{k_T(v)}$ with law $\eta^{(k_T(v))}$, where $k_T(v)$ denotes the number of child of $v$ in $T$. Label the edges connecting $v$ to its children using the $k_T(v)$ components of $[v]_T$ in left-to-right order, as prescribed by the plane structure of $T$. The label $\ell_T(u)$ of a vertex $u \in V(T)$ is defined as the sum of the edge labels along the ancestral path from the root (labelled $0$ by convention) to $u$. We denote by $\Pi_0$ the law of the resulting random labelled (plane) tree $(T, \ell_T)$. For $V \geq 1$, we also consider the law $\Pi_0^{(V)}$, which corresponds to the distribution of the same random tree conditioned to have $V$ edges. This class of models encompasses, for example, branching random walks, embedded complete and incomplete binary and ternary trees, random Cayley and plane trees with uniform displacements along edges. These models of labelled trees have become staple objects in probability theory and have been extensively studied both in combinatorics \cite{BousquetMelou, BousquetChapuy}, and in probability and random geometry \cite{AldousISE, Chauvin, DevroyeGW}, motivated by numerous applications to other models in statistical physics \cite{CoriVauquelin, BrownianMapMiermont, Schaeffer}.\\
	
	In this paper, we are interested in a process closely related to the so-called \emph{vertical profile} of random labelled trees. The vertical profile of a labelled tree $(t, \ell_t)$ is the family of nonnegative integers $(x_t(k))_{k\in \mathbb Z}$, where for every $k\in \mathbb Z$, $x_t(k)$ denotes the number of vertices in $(t, \ell_t)$ with label $k$. For a randomly sampled tree, the vertical profile describes the (random) distribution of vertex labels, and can be interpreted as the occupation measure of the branching process encoded by the random labelled tree. In particular, some work has been done in trying to understand the asymptotic properties of the vertical profile when the size of the tree is conditioned to grow to infinity. Under the assumption that the offspring distribution is critical and the displacement laws are centred, it has been shown in various cases \cite{MelouJanson, DevroyeJanson} that the vertical profile converges (after suitable rescaling) to a random, differentiable function $f_{\mathtt{ISE}}$. This function is the density of the Integrated SuperBrownian Excursion, a random measure first introduced by Aldous in \cite{AldousISE}. As the name suggests, this density is connected to the study of the local times of the superBrownian motion---a random measure-valued stochastic process--- that has attracted interest in past decades \cite{ SuperBrowMotionKonne, SuperBrownMotionReimers, Sugitani}. Recent work has aimed to better understand the distributional properties of the random function $f_{\mathtt{ISE}}$ and related processes. These investigations use both continuous methods (such as e.g. excursion theory for Brownian motion indexed by the Brownian tree \cite{Abraham, LeGallRiera}) as well as discrete/combinatorial methods \cite{ BousquetMelou, MelouJanson, MarckertRotation}.\\
	
	  In a recent work by Chapuy and Marckert \cite{chapuy2022notedensityiserelated}, the authors study the uniform naturally embedded incomplete binary tree model and show that the 3D process formed by the vertical profile, its discrete derivative, and its discrete integral behaves—when properly conditioned—like a Markov chain. Based on this, they conjectured that the limiting process $f_{\mathtt{ISE}}$, together with its derivative and integral, forms a Markov process governed by an SDE. This conjecture was confirmed for a closely related process in \cite{SDELocalTimesI, SDELocalTimes2}. The results in \cite{chapuy2022notedensityiserelated} only work for the binary tree model, as they relied heavily on explicit enumeration formulas for binary trees with prescribed vertical profile, established in \cite{BousquetChapuy} by Chapuy and Bousquet-Mélou. Nonetheless, these results suggest that the vertical profile could be studied effectively via Markov chain methods, an insight that partly motivated the present work.

	   The main contribution of this paper is to show that, by slightly shifting focus--- namely by considering a process counting \emph{edges} between given labels rather than \emph{vertices} with given labels---we recover a Markov structure that holds for a much broader class of models. 
	    Specifically, for every \(k \geq 1\), let \(X_k^+\) denote the number of oriented edges in \(T\) from a vertex labelled \(k-1\) to a vertex labelled \(k\),  where by convention edges are oriented away from the root. Similarly let \(X_k^-\) denote the number of oriented edges from a vertex labelled \(k\) to a vertex labelled \(k-1\). The resulting two dimensional process $(X^+_k, X^-_k)_{k\geq 1}$ can be thought of as the "\emph{vertical edge profile}" of the tree. Note that when the label increments along edges are constrained to lie in $\{-1, 1\}$, this encodes more information than the classical vertical profile since:
	\begin{align*}
	\forall k\geq 1, \ 	x_T(k)= X_k^++ X_{k+1}^-.
	\end{align*}
	We prove that, under the  unconditioned law $\Pi_0$, the process $(X_k^+, X_k^-)_{k\geq 1}$ is a time homogeneous Markov chain for the full class of models introduced earlier. Moreover, under the conditioned probabilities $\Pi_0^{(V)}$, we prove an analogous result: if we enrich the process by adding the number of edges below label $k$, we can recover a (three-dimensional) Markov chain, reminiscent of the one identified in \cite{chapuy2022notedensityiserelated}. \\
	
	To the best of our knowledge, this Markov property has not been previously observed, at least in this generality. For the binary tree model, we go further and provide a closed-form expression for the transition kernel of the associated Markov chain. This is closely related to existing product-form enumeration formulas for the number of binary trees with a given vertical edge profile, established in \cite{BousquetChapuy}. Our derivation of the kernel is independent of these results and may offer a new, more probabilistic perspective on these enumeration formulas—at least in the binary case. More generally, if similar closed-form kernels could be identified for other tree models, this could open new directions in the combinatorial enumeration of (possibly weighted) labelled trees with prescribed vertical edge profiles. \\
	
		Our analysis relies on a decomposition of labelled trees into \emph{tree excursions} above and below a fixed level $m\geq 1$, which we obtain by cutting all edges connecting vertices labelled \(m-1\) and \(m\) (in either direction). More precisely, let $T$ be a random labelled tree and fix a label $m\geq 1$. Consider an edge $e$ connecting two vertices labelled $m$ and $m-1$, and write $u \prec v$ for its endpoints, with $u$ the parent of $v$. We disconnect $e$ by duplicating the vertex $v$, assigning one copy of $v$ to each of the two resulting components. One component will consist of $v$ and all its descendants, while the other one will contain all remaining vertices, including the duplicate of $v$. Repeating this operation for each edge connecting vertices labelled $m$ and $m-1$ yields a forest of labelled plane trees. Among these, the trees rooted at a label $m$ have all their labels $\geq m-1$, and we refer to them as tree excursions above level $m$. Conversely, the trees rooted at a label $m-1$ have all their labels $\leq m$, and we call them the tree excursions below level $m$.
		 If we retain the information of how these excursions were connected in the original tree, we obtain a plane tree structure on the forest of excursions. This decomposition is key to proving the Markov properties of the vertical edge profile.  It also lays useful groundwork for the companion paper \cite{Companion}, in which we will establish scaling limits for the Markov chain $(X^+_m, X^-_m)_m$ (see the next subsection for details).\\
	 
	 The main contributions of this paper can be summarized as follows:
	\begin{itemize}
		\item[\textbullet]\textbf{Excursion forest decomposition.} We define the excursion forest decomposition and describe its law (\emph{Proposition~\ref{distributionForest}}). For several models, we also derive explicit generating functions for the offspring distributions of the excursion forests (\emph{Proposition~\ref{ExplicitForm}}).

		\item[\textbullet] \textbf{Markov property (unconditioned law).} We identify the conditional law of the tree excursions above level $m$ knowing $X_m^+$ and $X_m^-$ and use this to show that the process \((X_m^{+}, X_m^{-})_{m \geq 1}\) is a time homogeneous Markov chain under $\Pi_0$ (\emph{Theorem~\ref{discreteMarkovProperty}}).
		
		\item[\textbullet] \textbf{Markov property (conditioned law).} Under the conditioned law $\Pi_0^{(V)}$, we prove an analogous result: the 3D process $(X_m^+, X_m^-, M_m^-)_{m\geq 1}$, where $M_m^-$ denotes the number of edges whose both endpoints have a label smaller than $m$, forms a Markov chain (\emph{Proposition~\ref{discreteMarkoCondiProp}}).
		
		\item[\textbullet] \textbf{Explicit kernels in the binary case.} For the binary tree model, we compute explicit expressions for the transition kernels of the Markov chains $(X_m^+, X_m^-)_m$ and  $(X_m^+, X_m^-, M_m)_m$ under $\Pi_0$ and $\Pi_0^{(V)}$ respectively (\emph{Theorem \ref{ThmExplicitKernel}, Theorem \ref{ThmExplicitKernelCondi}}).These kernels are closely related to combinatorial formulas established in \cite{BousquetChapuy}, although our derivation is independent and offers a more probabilistic viewpoint. We explain in some detail the connection between the two approaches.
		
		\item[\textbullet] \textbf{Application to random maps.} As a motivating application, we establish a Markov property for the growth of the perimeter and number of connected components of the complement of balls in random Boltzmann quadrangulations (\emph{Proposition~\ref{AppliRandomMaps}}). A similar result was known in the continuous setting for the Brownian map (cf Proposition~11 in \cite{SDELocalTimesI}).
	\end{itemize}

	\subsection{Links with super-Brownian motion occupation measure}

	The fact that $(X^{+}_m, X^{-}_m)_m$ is Markovian is reminiscent of a theorem recently proved by Le Gall in \cite{SDELocalTimesI}, showing a Markov property related to the occupation measure of super-Brownian motion. Let us recall the statement and a bit of historical background behind this result. Write $(L_x^\infty)_{x\in \mathbb R}$ for the density of the total occupation measure of super-Brownian motion started at the Dirac measure $\delta_0$ (we refer to \cite{SDELocalTimesI, Sugitani} for a precise definition of this process). A striking property of the process \((L_x^\infty)_{x \geq 0}\) is the fact that it admits an almost surely continuous derivative \((\dot{L}_x^\infty)_{x \geq 0}\). This was first shown by Sugitani \cite{Sugitani} in the context of super-Brownian motion, and rediscovered via discrete models in later works \cite{MelouJanson, chapuy2022notedensityiserelated}. In \cite{SDELocalTimes2}, Le Gall established that the process \((L_t^\infty, \dot{L}_t^\infty)_{t \geq 0}\) is Markovian. This was further strengthened by Le Gall and Perkins \cite{SDELocalTimes2}, who derived the following stochastic differential equation:
	\begin{align*}
		\mathrm{d} \dot{L}_t^\infty = 4 \sqrt{L_t^\infty} \, \mathrm{d}B_t + g\left(L_t^\infty, \frac{1}{2} \dot{L}_t^\infty \right) \mathrm{d}t.
	\end{align*}
	The drift term involves the function 
	$g(t, y) = 8tp_t'(y)/p_t(y)$
	where \(p_t\) is the density at time \(t\) of a stable Lévy process with index \(3/2\), characterized by its Laplace exponent \(\Psi(\lambda) = \sqrt{2/3} \, \lambda^{3/2}\) and $p_t'$ is the derivative of $y\mapsto p_t(y)$. Interestingly, we show in this work that the closed-form expression for the transition kernel of the Markov chain \((X_m^+, X_m^-)_m\) in the binary tree model (Theorem~\ref{ThmExplicitKernel}) involves the distribution of a random walk in the domain of attraction of a \(3/2\)-stable law. 	This connection naturally leads to the question of whether the discrete two-dimensional process \((X^{+}_m, X^{-}_m)_{m}\) is related in the limit, after rescaling, to the continuous pair \((L_x^\infty, \dot{L}_x^\infty)_{x \geq 0}\). In the companion paper \cite{Companion}, we will show that this is indeed the case. 	More precisely, let us assume that the offspring distribution $\xi$ is critical, and that label increments along edges are independent and uniformly distributed in $\{-1, 1\}$. Consider a family of independent labelled trees $(T_k)_{k\geq 1}$ sampled according to the law $\Pi_0$, and for every $n$, let $F_n = (T_1, \dots, T_n)$ denote the forest associated with the first $n$ trees.
	For each \(k \geq 1\), define \(X_k^{+, n}\) (resp.\ \(X_k^{-, n}\)) as the total number of edges in \(F_n\) oriented from a vertex labelled \(k-1\) to a vertex labelled \(k\) (resp.\ from \(k\) to \(k-1\)). We show in \cite{Companion} that we have:
	\begin{align*}
		\left(
		\frac{X^{+, n}_{\lfloor \gamma \sqrt{n} x \rfloor} + X^{-, n}_{\lfloor \gamma \sqrt{n} x \rfloor}}{\gamma n^{3/2}}, \quad
		\frac{X^{+, n}_{\lfloor \gamma \sqrt{n} x \rfloor} - X^{-, n}_{\lfloor \gamma \sqrt{n} x \rfloor}}{n}
		\right)
		\xrightarrow[n \to \infty]{(d)}
		\left(L_x^\infty, \tfrac{1}{2} \dot{L}_x^\infty\right),
	\end{align*}
	where \(\gamma = 2 / \sigma_\xi\), with \(\sigma_\xi^2\) the variance of the offspring distribution \(\xi\).
	\section{Preliminaries}
	\subsection{General notations}
	
	Let us begin by introducing some general notations. A \emph{rooted tree} $T$ is a connected, acyclic graph on a finite vertex set $V(T)$, with a distinguished vertex $\rho_T \in V(T)$ called the \emph{root} of $T$. We say that $T$ is a (rooted) plane tree if in addition, for each vertex $v \in V(T)$, a linear order is specified on the set of its children in $T$—that is, the (possibly empty) set of vertices adjacent to $v$ that are farther from $\rho_T$ than $v$ with respect to the graph distance in $T$.
	
	For each $v \in V(T)$, we write:
	\begin{itemize}
		\item[\textbullet] $k_T(v)$ for the number of children of $v$,
		\item[\textbullet] $\pi_T(v)$ for the parent of $v$ in $T$, defined only if $v \ne \rho_T$.
	\end{itemize}
	By convention, the edges of $T$ are oriented away from the root—that is, from parents to children. This orientation induces a partial order $\preceq$ on the vertices of $T$, where $u \preceq v$ if and only if $u$ lies on the ancestral path from the root to $v$. We write $u \prec v$ to mean $u \preceq v$ and $u \ne v$. In the following, the \emph{subtree branching off} a vertex $v$ is defined as the subtree of $T$ rooted at $v$, consisting of all descendants of $v$ in $T$—that is, the subtree induced by all vertices $u \in V(T)$ such that $v \preceq u$.
	
	In this work, we are interested in models of \emph{labelled} rooted plane trees, defined as pairs $(T, \ell_T)$, where $T$ is a rooted plane tree and $\ell_T : V(T) \to \mathbb{Z}$ is an integer-valued function assigning a label to each vertex of $T$, subject to the constraint:
	\begin{align*}
		\forall v \in V(T) \setminus \{\rho_T\}, \quad \ell_T(v) - \ell_T(\pi_T(v)) \in \{-1, 0, 1\}.
	\end{align*}
	These definitions naturally extend to random labelled forests as well—that is, finite sequences of labelled rooted plane trees. In the following, we write $\mathbb T$ for the set of all labelled plane trees. If $t\in \mathbb T$, we will (somewhat abusively) keep the same notation $t$ to talk about the associated rooted plane tree (so that the notation $V(t), \rho_t,  k_t(v), \pi_t(v)\cdots$ still make sense) and we will always write $\ell_t:V(t)\to \mathbb Z$ for the associated labelling function. If $x\in \mathbb Z$, we will write $\mathbb T_x$ for the set of all labelled trees rooted at $x$:
	\begin{align*}
		\mathbb T_x= \{ t\in \mathbb T \mid \ell_t(\rho_t)=x\}.
	\end{align*}
	
	\noindent \textbf{Truncation operation.} Let $t\in \mathbb T$ and fix $l\in \mathbb Z$. The \emph{truncation} of $t$ at level $l$, denoted by $t^{[l]}$, is obtained from $t$ by deleting all vertices of $t$ with a strict ancestor of label $l$. In other words, $t^{[l]}$ is the subtree of $t$ whose vertex set is:
	\begin{align*}
		\Big\{u\in V(t) \;\mid\; \forall v\prec u,\ \ell_t(v)\neq l \Big\}.
	\end{align*}
	We view $t^{[l]}$ as a labelled tree (where vertices obviously keep the same labels as in $t$).
	We next define \emph{tree excursions}, which correspond to truncations at level $0$ of trees rooted at $1$ or $-1$, and which play an important role in the following.
	\begin{definition}
		\label{Def1}
		A \emph{positive tree excursion} (resp. \emph{negative tree excursion}) is a labelled plane tree rooted at label $1$ (resp. label $-1$), in which all labels are nonnegative (resp. nonpositive), and such that each vertex with label $0$ is a leaf. We denote by $\mathcal E_+$ (resp. $\mathcal E_-$) the set of all positive (resp. negative) tree excursions.
		If $\tau$ is a tree excursion, we define 
		\begin{align*}
			n_\tau = \text{Card}\{v \in V(\tau) \mid \ell_\tau(v) = 0\},
		\end{align*}
		the number of vertices with label $0$ in $\tau$.
	\end{definition}
	
	\subsection{The excursion forest decomposition}
	\label{Defdecomp}
	Let $T\in \mathbb T_0$ be a labelled plane tree rooted at $0$, and fix an integer $m \in \{1, 2, \cdots\}$. In this section, we describe a natural decomposition of the tree $T$ into tree excursions occurring above and below level $m$.
	
	We begin by disconnecting all edges in $T$ that cross level $m-1/2$ (i.e. with one endpoint labelled $m-1$ and the other one labelled $m$). More precisely, for each such edge $e$, let $u \prec v$ denote its endpoints, with $u$ being the parent of $v$. We disconnect the tree at $v$ by duplicating the vertex $v$ and separating the two resulting copies. One component will consist of $v$ and its descendants, while the other one will contain all remaining vertices, including the duplicate of $v$. By performing this operation for every edge connecting a vertex labelled $m-1$ to one labelled $m$ (or conversely), we obtain a forest of labelled plane trees. Consider a component $\tau$ of this forest, and let $\rho_\tau$ denote the smallest vertex of $\tau$ with respect to the order $\preceq$. We view $\tau$ as a labelled plane tree with root identified to $\rho_\tau$. There are three possibilities: 
	\begin{enumerate}
		\item We have $\rho_\tau = \rho_T$. This occurs for a unique component, which coincides with $T^{[m]}$ (recall the definition of the truncation operation from the previous section). We refer to this component as the \emph{root component of $T$ truncated at level $m$}.
		\item We have $\ell_T(\rho_\tau) = m$. Subtracting $m-1$ from each label in $\tau$ (i.e. we set $\ell_\tau(v)= \ell_T(v)-m+1$ for every $v\in V(\tau)$) yields a positive tree excursion, as defined in Definition~\ref{Def1}. We denote the unordered collection (or multiset) of all positive tree excursions obtained this way by $\{\tau_1^{m, +}, \dots, \tau_{X_m^+}^{m, +}\}$. 
		\item We have $\ell_T(\rho_\tau) = m-1$ (when $m=1$ we exclude the case $\rho_\tau = \rho_T$ which has already been considered). Subtracting $m$ from each label in this component yields a negative tree excursion. The corresponding (unordered) collection of negative tree excursions is denoted by $\{\tau_1^{m, -}, \dots, \tau_{X_m^-}^{m, -}\}$.
	\end{enumerate}
	
	\begin{figure}[h]
		\centering
		\begin{minipage}[b]{0.45 \textwidth}
			\includegraphics[width=\textwidth]{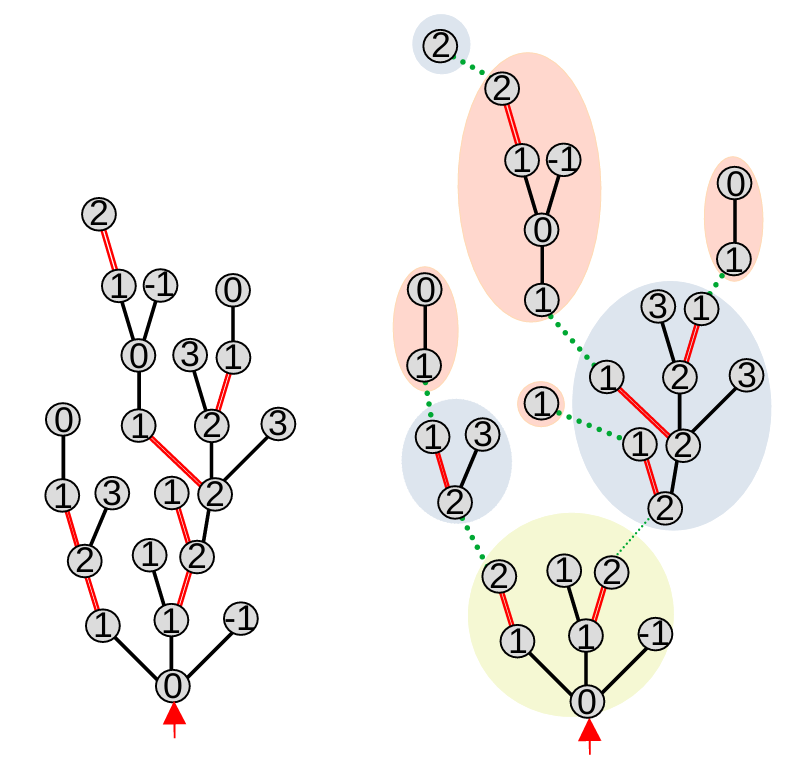}
		\end{minipage} \vrule
		\begin{minipage}[b]{0.46\textwidth}
			\includegraphics[width=\textwidth]{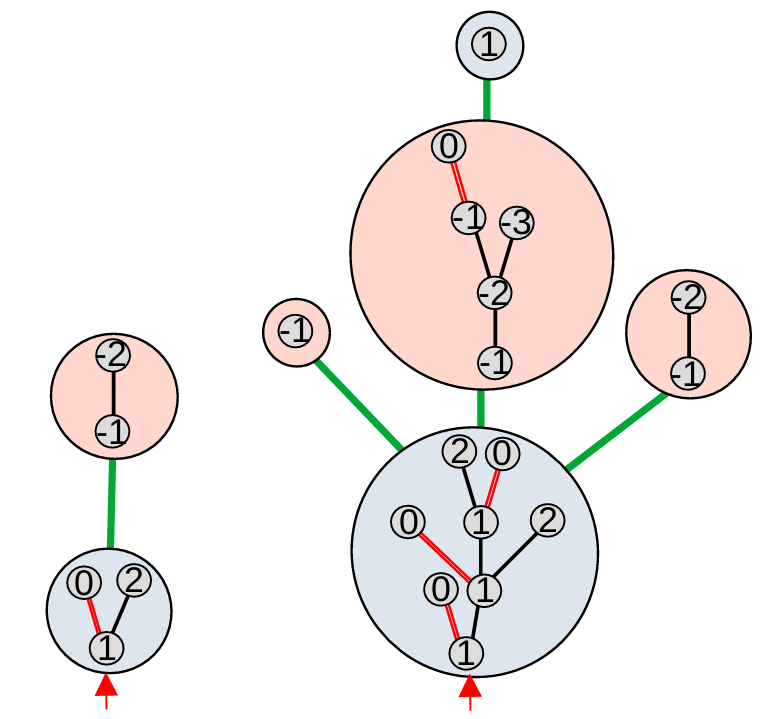}
		\end{minipage}%
		\caption{Example of the construction of the excursion forest for $m=2$. Consider a labelled plane tree rooted at $0$ (left), edges between labels $1$ and $2$ are in red in the figure. Disconnect each of these edges to get a forest of labelled rooted plane (sub) trees. The associated excursion forest at level $2$ (right) has one vertex per component of this forest that does not contain the root.}
		
	\end{figure}
	
	\noindent By construction, the number \( X_m^+ \) of positive tree excursions  obtained in case $2$ is equal to the number of oriented edges in \( T \) from a vertex labelled $m-1$ to a vertex labelled $m$. Similarly, the number \( X_m^- \) of negative excursions corresponds to the number of oriented edges in \( T \) from a vertex labelled $m$ to a vertex labelled $m-1$. The sequence $(X_m^+, X_m^-)_{m=1, 2, \cdots}$ is the main object of interest in this work.
	
	The set $\{\tau_i^{m, -}\}_{i\leq X_m^-}\cup\{\tau_j^{m, +}\}_{j\leq X_m^+}$ of all excursions inherits a genealogical structure from the initial genealogical order of $T$. We now define a forest that encodes this structure. Informally, this forest is obtained by removing the vertices of the root component $T^{[m]}$ and identifying vertices that belong to the same excursions, as defined in the previous decomposition.
	\begin{definition}[Excursion Forest]
		The \emph{excursion forest} $\mathcal{F}^m_T$ at level $m$ associated with a tree $T$ is the forest whose vertices are the tree excursions at level $m$, namely the collection $\{\tau_i^{m, -}\}_{i \leq X_m^-} \cup \{\tau_j^{m, +}\}_{j \leq X_m^+}$. In $\mathcal{F}^m_T$, a tree excursion $\tau$ is the parent of another excursion $\tau'$ if and only if the root of $\tau'$ corresponds to a leaf of $\tau$ in the original tree $T$.
		
		The forest $\mathcal{F}^m_T$ inherits a planar structure from $T$: its components are ordered, and each tree in $\mathcal{F}^m_T$ carries a planar embedding. Each component of $\mathcal{F}^m_T$ is rooted at a positive tree excursion, and along any ancestral line in $\mathcal{F}^m_T$, positive and negative excursions alternate.
	\end{definition}
	
	\noindent
	Equivalently, one may view $\mathcal{F}^m_T$ as a plane forest in which each vertex is decorated by a tree excursion. We denote by $\tilde{\mathcal{F}}^m_T$ the underlying plane forest obtained by forgetting the excursions labelling each vertex. 	
	Note moreover that the number of components of $\mathcal{F}^m_T$ equals the number of leaves labelled $m$ in the subtree $T^{[m]}$. The construction should be clear from Figure $1$.
	
	A crucial feature of this construction is the fact that the original tree $T$ can be fully recovered from the triple $(T^{[m]}, \tilde{\mathcal{F}}^m_T, (t_v)_{v \in V(\tilde{\mathcal{F}}^m_T)})$, where $(t_v)$ is the collection of tree excursions labelling the vertices of the forest. Let us make this more precise. Fix a plane forest $F$. A collection $(t_v)_{v \in V(F)} \in (\mathcal{E}_+ \cup \mathcal{E}_-)^{V(F)}$ of tree excursions indexed by the vertices of $F$ is called \emph{admissible} if for every $v \in V(F)$:
	\begin{enumerate}
		\item[\textbullet] $n_{t_v} = k_F(v)$, where $k_F(v)$ denotes the number of children of $v$ in $F$, and $n_{t_v}$ the number of leaves labelled $0$ in $t_v$;
		\item[\textbullet] $t_v$ is a positive tree excursion if $v$ is at even height in $F$, and a negative tree excursion if $v$ is at odd height.
	\end{enumerate}
	We write $M(F)$ for the set of all admissible collections $(t_v)_{v \in V(F)}$.
	
	Now let $n \geq 0$, and let $t_\rho \in \mathbb{T}_0$ be a labelled plane tree with $n$ leaves labelled $m$, such that all other vertices have a label (strictly) smaller than $m$. Let $F$ be a bicoloured plane forest with $n$ components, and choose an admissible collection $(t_v)_{v \in V(F)} \in M(F)$. Then there exists a unique labelled plane tree $T \in \mathbb{T}_0$ such that:
	\[
	T^{[m]} = t_\rho \quad \text{and} \quad \mathcal{F}_T^m = (F, (t_v)_{v \in V(F)}),
	\]
	i.e., the excursion forest of $T$ corresponds to the plane forest $F$ marked by the excursions $(t_v)_{v\in V(F)}$. In the following, we will denote the tree reconstructed this way by:
	\begin{align}
		\label{Phim}
		T = \Phi^m(t_\rho, F, (t_v)_{v \in V(F)}).
	\end{align}
	
	If $m\in\{-1, -2, \cdots\}$, we can define symmetrically the forest of excursions above and below level $m$ by disconnecting all edges that connect a vertex labelled $m$ to a vertex labelled $m+1$. The root component is again defined as $T^{[m]}$ and we obtain a set of positive and negative tree excursions just as before. We write $\check X^+_{-m}$ for the number of positive tree excursions obtained this way and $\check X^-_{-m}$ for the number of negative tree excursions. Equivalently, $\check X_{m}^+$ is the number of oriented edges that connect a vertex labelled $-m$ to a vertex labelled $-m+1$ and $\check X_m^-$ is the number of oriented edges that connect a vertex labelled $-m+1$ to a vertex labelled $-m$. 
	
	\subsection{Models of random trees}
	We now introduce the models of random labelled trees we are interested in this work. They correspond to labelled Galton--Watson trees $T$ with a given critical or subcritical offspring distribution $\xi$, such that the vector of displacements of the children of each vertex $v$ follows a fixed displacement law that may depend on the number of children of $v$. More precisely, we consider a probability measure $\xi$ on $\{0, 1, 2, \cdots\}$ such that
	\begin{align*}
		\sum_{k\geq 0} \xi(k)k\leq 1,
	\end{align*}
	and, for every $k\geq 1$, a probability measure $\eta^{(k)}$ on $\{-1, 0, 1\}^k$. Let $t$ be a labelled plane tree with labelling function $\ell_t$. For every $v\in V(t)$, we write:
	\begin{align*}
		[v]_t=\Big(\ell_t(u_1)-\ell_t(v), \cdots, \ell_t(u_{k_t(v)})-\ell_t(v)\Big),
	\end{align*} 
	where $u_1, \cdots, u_{k_t(v)}$ are the children of $v$ in left to right order. In other words, $[v]_t$ is the vector of label increments along the edges from $v$ to its children. When $v$ is a leaf, $[v]_t$ is by convention the empty array $\varnothing$ and we set in this case $\eta^{(0)}(\varnothing)=1$. We then define for every labelled plane tree $t$:
	\begin{align*}
		\Pi(t)= \prod_{v \in V(t)} \Big( \xi(k_t(v)) \cdot \eta^{(k_t(v))}([v]_t)\Big).
	\end{align*}
	\begin{definition}
		\label{Pi0}
		For every $x \in \mathbb{Z}$, the law $\Pi_x$ of a labelled Galton--Watson tree rooted at $x$, with offspring distribution $\xi$ and spatial displacement distributions $(\eta^{(k)})_{k\geq 0}$, is the probability measure corresponding to the restriction of $\Pi$ to the set $\mathbb T_x$ of labelled plane trees rooted at $x$.
	\end{definition}
	
	\noindent For all $x\in \mathbb Z$ and $l\neq x$, we will write $\Pi_x^{[l]}$ for the law of $T^{[l]}$ when $T$ is sampled according to $\Pi_x$ (recall the definition of the truncation $T^{[l]}$ from the previous section).
	To simplify notations, we write $\Pi^+= \Pi^{[0]}_1$ and $\Pi^-= \Pi^{[0]}_{-1}$. The support of the measure $\Pi^+$ (resp. $\Pi^-$) is contained in $\mathcal E_+$ (resp. $\mathcal E_-$). Moreover, a straightforward computation shows that if $\tau\in \mathcal E_+$ is a positive tree excursion, then
	\begin{align}
		\label{Piplus}
		\Pi^+(\tau) = \prod_{\substack{v \in V(\tau) \\ \ell_\tau(v) \neq 0}}\left( \xi(k_{\tau}(v)) \cdot \eta^{(k_\tau(v))}\left([v]_\tau\right)\right) ,
	\end{align}
	and the analogous expression holds for $\Pi^-$ when $\tau\in \mathcal E^-$.
	
	Let $m\geq 1$. Fix $T\in \mathbb T_0$ and recall the decomposition of $T$ in the root component $T^{[m]}$ and the tree excursions $(\tau_k^{m, +})_{1\leq k \leq X_m^+}$ and $(\tau_k^{m, -})_{1\leq k\leq X_m^-}$ from the previous section. On the set of vertices, this decomposition corresponds to the following partition of $V(T)$:
	\begin{align*}
		V(T)=	\{u\in V(T^{[m]}) \mid \ell_T(u)\neq m\}  &\cup\bigcup_{k\leq X_m^+} \left\{u\in V(\tau_k^{+, m}) \mid \ell_{\tau_k^{+, m}}(u)> 0\right\} \\
		&\cup \bigcup_{j\leq X_m^-} \left\{u\in V(\tau_j^{-, m}) \mid \ell_{\tau_j^{-, m}}(u)<0 \right\}.
	\end{align*}
	In particular, we can compute:
	\begin{align*}
		\Pi_0(T)= \prod_{\substack{v \in V(T^{[m]})\\ \ell_{T}(v)\neq m}} &\Big( \xi(k_T(v)) \cdot \eta^{(k_T(v))}([v]_T)\Big) \times \prod_{k\leq X_m^+} \prod_{\substack{u\in V(\tau_k^{+, m})\\ \ell_{\tau_k^{+, m}}(u)> 0}}\Big( \xi(k_T(u)) \cdot \eta^{(k_T(u))}([u]_T)\Big)  \\ & \times \prod_{j\leq X_m^-} \prod_{\substack{u\in V(\tau_j^{-, m}) \\ \ell_{\tau_j^{-, m}}(u)<0 }} \Big( \xi(k_T(u)) \cdot \eta^{(k_T(u))}([u]_T)\Big).
	\end{align*}
	Thanks to \eqref{Piplus}, this simplifies to:
	\begin{align}
		\label{FormulaPiPhi}
		\Pi_0(T)=\Pi^{[m]}_0(T^{[m]})\times \prod_{k\leq X_m^+} \Pi^+(\tau_j^{m, +})\times  \prod_{j\leq X_m^-} \Pi^-(\tau_j^{m, -}).
	\end{align}
	\section{The law of the excursion forest}
	\subsection{Branching property of the excursion forest}
	
	In the following, we let $T$ be a random labelled plane tree rooted at $0$ sampled according to the law $\Pi_0$ introduced in definition \ref{Pi0}. We now turn to studying the law of the excursion forests $\mathcal{F}^{k}_T$, for $k\in \mathbb Z\setminus \{0\}$.
	If $k\geq 1$ is an integer, let us write $N_k$ for the number of trees in the forest $\tilde{\mathcal F}_T^{k}$ and recall that $N_k$ also corresponds to the number of leaves labelled $k$ in the root component $T^{[k]}$. Symmetrically, write $\check N_{k}$ for the number of trees in the forest $\tilde{\mathcal F}_T^{-k}$, which corresponds to the number of leaves labelled $-k$ in the root component $T^{[-k]}$. Let $\nu_-$ denote the law of $N_{1}$. By construction, $N_{1}$ is the number of leaves with label $1$ in the tree $T^{[1]}$. If we shift each label of $T^{[1]}$ by $-1$, this yields a negative tree excursion with law $\Pi^-$. It follows that $N_{1}$ has the same law as $n_t$ when $t$ is sampled according to $\Pi^-$. In particular:
	\begin{align}
		\label{Nuplus}
		\nu_-(k)= \sum_{t\in \mathcal E_- ; \ n_t=k} \Pi^-(t),
	\end{align}
	where the sum is over all negative tree excursions with exactly $k$ vertices labelled $0$. Similarly, if $\nu_+$ is the law of $\check N_{1}$, we have:
	\begin{align}
		\label{Numoins}
		\nu_+(k)= \sum_{t\in \mathcal E_+; \ n_t=k} \Pi^+(t).
	\end{align}
	We set by convention $N_0=\check N_0=1$. 
	
	\begin{lemma}
		\label{lemmanu}
		The processes $(N_k)_{k\geq 0}$ and $(\check N_{k})_{k\geq 0}$ are Galton--Watson processes started at $1$, with respective offspring distribution $\nu_-$ and $\nu_+$.
	\end{lemma}
	\begin{proof}
		We prove the lemma for $(N_k)_k$, and the argument works symmetrically for $(\check N_k)_k$.
		For a labelled plane tree $t$, and an integer $k>0$ let
		\begin{align*}
			\mathcal N_k(t)= \left\{u\in V(t) \mid \ell_t(u)= k \text{ and } \forall v\prec u, \ell_t(v)\neq k\right\},
		\end{align*} 
		the set of all vertices which are \emph{first hitting points of $k$}. Clearly $N_{k}= \text{Card}(\mathcal N_{k}(T))$. Fix $k\geq 1$ and write $T_1^{k}, \cdots, T_{N_{k}}^{k}$ for the sub-trees of $T$ branching off the vertices of $\mathcal N_{k}(T)$ listed in the planar ordering. Subtract $k$ from all their labels so that the $T_i^{k}$ can be considered as elements of $\mathbb T_0$. Conditionally on $(N_{1}, \cdots N_{k})$, the trees $T_1^{k}, \cdots, T_{N_{k}}^{k}$ are distributed as $N_{k}$ independent trees with law $\Pi_0$. This follows by a standard application of the Markov property of stopping lines for (labelled) Galton--Watson trees --- see e.g. Proposition $2.1$ in \cite{Chauvin} for a similar application. If $k\geq 1$, we have clearly
		\begin{align*}
			\mathcal N_{k+1}(T)=\bigsqcup_{i\leq N_{k}} \mathcal N_1(T_{i}^{k}),
		\end{align*}
		and it follows that conditionally on $(N_1, \cdots, N_{k})$, $N_{k+1}$ is distributed as the sum of $N_{k}$ independent copies of $N_{1}$. This shows that $(N_{k})_{k\geq 0}$ is a Galton--Watson process started at $1$, with offspring distribution $\nu_-$, proving the lemma.
	\end{proof}
	\begin{remark}
		Note that the forest $\mathcal F^m_T$ is a measurable function of the trees $T_1^m,\cdots,T_{N_m}^m$. In particular, it follows from the stopping line property invoked in the proof of Lemma \ref{lemmanu} that conditionally on $N_m$, the forest $\mathcal F^m_T$ is independent of $(N_l)_{l\leq m}$ (and of the root component $T^{[m]}$).
	\end{remark}
	
	We now proceed to describe the law of ${\mathcal F}_T^m$.
	Given two offspring distributions $\mu_-$, $\mu_+$, we define an alternating two-type Galton--Watson tree with offspring distributions $(\mu_-, \mu_+)$ by requiring that the offspring distribution is $\mu_+$ for vertices at even generation and $\mu_-$ for vertices at odd generation. More precisely, for a plane tree $t$ and $v\in V(t)$ write $\epsilon_t(v)=+$ if $v$ is at even height in $t$ and $\epsilon_t(v)=-$ if $v$ is at odd height. The law of an alternating two-type Galton--Watson tree $\mathcal T$ with offspring distributions $(\mu_-, \mu_+)$ is given for any plane tree $t$ by:
	\begin{align*}
		\mathbb P(\mathcal T= t)= \prod_{v\in V(t)} \mu_{\epsilon_t(v)} (k_t(v)).
	\end{align*}
	Recall that $\tilde {\mathcal F}^m_T$ is the plane bicoloured forest obtained from $\mathcal F^m_T$ by forgetting the excursions associated to the vertices of $\mathcal F^m_T$.
	
	\begin{proposition}
		\label{distributionForest}
		Let $m\geq 1$, conditionally on $N_m$, the forest ${\tilde {\mathcal F}^m_T}$ has the law of $N_m$ independent alternating two-type Galton--Watson trees with offspring distributions $(\nu_+, \nu_-)$.

		Conditionally on $\{\tilde {\mathcal F}_T^m=F\}$, for some plane forest $F$, the random tree excursions $(t_v)_{v\in V(\tilde{\mathcal F}_T^m)}$ associated to the vertices of $\tilde {\mathcal F}_T^m$ are independent, and, for every vertex $v$ of $\tilde {\mathcal F}^m_T$, the (conditional) law of the tree excursion $t_v$ is given by: 
		\begin{align*}
			\Pi^{\epsilon_{F}(v)}\Big( \ \cdot \ \Big| \ n_\tau= k_F(v)\Big).
		\end{align*}
	\end{proposition}
	\begin{proof}
		Let $n\geq 0$ and fix a labelled plane tree $t_\rho\in \mathbb T_0$ with $n$ leaves labelled $m$ and such that labels of other vertices (including other leaves) are smaller than $m$. Let $F$ be a plane forest with $n$ trees. Recall the discussion at the end of Section $2.2$, and in particular the notation $\Phi^m$ in (\ref{Phim}). We have:
		\begin{align*}
			\Pi_0\left(T^{[m]}=t_\rho; \tilde {\mathcal F}_T^m= F\right) &= \sum_{(t_v)_{v\in V(F)}\in M(F)} \Pi_0\left(\Phi^m(t_\rho, F, (t_v)_{v\in V(F)})\right),
		\end{align*}
		where we recall the definition of the set $M(F)$ of all admissible collections $(t_v)_{v\in V(F)}$ of tree excursions. Using (\ref{FormulaPiPhi}), we have for every $(t_v)_{v\in V(F)}\in M(F)$:
		\begin{align}
			\label{*}
			\Pi_0\Big(\Phi(t_\rho, F, (t_v)_{v\in V(F)})\Big)= \Pi_0^{[m]}(t_\rho)\prod_{v\in V(F)} \Pi^{\epsilon_F(v)}(t_v),
		\end{align}
		so that in particular:
		\begin{align*}
			\Pi_0\left(T^{[m]}=t_\rho; \tilde {\mathcal F}_T^m= F\right)	&=\sum_{(t_v)_{v\in V(F)}\in M(F)} \Pi_0^{[m]}(t_\rho)\prod_{v\in V(F)} \Pi^{\epsilon_F(v)}(t_v)\\
			&= \Pi_0^{[m]}(t_\rho)\prod_{v\in V(F)} \left(\sum_{t\in \mathcal E_{\epsilon_F(v)} \ :\  n_{t}= k_F(v)} \Pi^{\epsilon_F(v)}(t) \right)\\
			& = \Pi_0^{[m]}(t_\rho)\prod_{v\in V(F)} \nu_{\epsilon_F(v)}(k_F(v)), 
		\end{align*}
		where we used formulas (\ref{Nuplus}) and (\ref{Numoins}) in the last line. If we sum over all choices of $t_\rho$, we get:
		\begin{align}
			\label{**}
			\Pi_0(N_m=n, \tilde{\mathcal F}_T^m=F)= \Pi_0(N_m=n)\prod_{v\in V(F)}\nu_{\epsilon_F(v)}(k_F(v)).
		\end{align}
		This implies that, conditionally on $N_m=n$, $\tilde {\mathcal F}_T^m$ is distributed as a forest of $n$ independent alternating two-type Galton--Watson trees with offspring distributions $(\nu_+, \nu_-)$.
		
		We turn to the second part of the proposition. Using (\ref{*}) and (\ref{**}), we get that conditionally on $\{\tilde{\mathcal F}_{T}^m=F, N_m=n\}$, the probability that the root component $T^{[m]}$ is equal to $t_\rho$ and that the tree excursions associated to the vertices of ${\mathcal F}_{T}^m$ exactly correspond to a fixed admissible collection $(t_v)_{v\in V(F)}\in M(F)$ is given by:
		\begin{align*}
			\frac{\Pi_0\left(\Phi^m\left(t_\rho, F, (t_v)_{v\in  V(F)}\right)\right)}{\Pi_0\left(N_m=n; \ \tilde {\mathcal F}_T^m=F \right)}	= \frac{\Pi_0^{[m]}(t_\rho)}{\Pi_0(N_m=n)} \prod_{v\in V(F)} \frac{\Pi^{\epsilon_F(v)}(t_v)}{\nu_{\epsilon_F(v)}(k_F(v))}.
		\end{align*}
		Note that the property $(t_v)_{v\in V(F)}\in M(F)$ implies $n_{t_v}= k_F(v)$ for every $v\in V(F)$, and:
		\begin{align*}
			\frac{\Pi^{\epsilon_F(v)}(t_v)}{\nu_{\epsilon_F(v)}(k_F(v))}=\Pi^{\epsilon_F(v)}\Big(t_v \ \Big| \ n_t=k_F(v)\Big).
		\end{align*} 
		This proves the proposition: if we condition on $\tilde {\mathcal F}_T^m= F$, the root component and the tree excursions $t_v$ marking the vertices of $\mathcal F_T^m$ are independent, and the law of the excursion $t_v$ marking the vertex $v\in V(F)$ is exactly given by $\Pi^{\epsilon_F(v)}\left(\cdot \mid n_t=k_F(v)\right)$. 
	\end{proof}
	
	\subsection{Explicit generating functions}
	
	In certain specific models, we can find explicit formulas for the offspring distributions $\nu_+$ and $\nu_-$. Let $g_{\nu_+}$ and $g_{\nu_-}$ denote the respective generating functions of $\nu_+$ and $\nu_-$. Using formulas (\ref{Nuplus}) and (\ref{Numoins}), we obtain:
	\begin{align*}
		g_{\nu_+}(z) &= \sum_{k \geq 0} \nu_+(k) z^k = \sum_{t \in \mathcal{E}^+} \Pi^+(t) z^{n_t}, \\
		g_{\nu_-}(z) &= \sum_{k \geq 0} \nu_-(k) z^k = \sum_{t \in \mathcal{E}^-} \Pi^-(t) z^{n_t}.
	\end{align*} 
	Let $T$ be a tree distributed according to $\Pi^+$. Denote by $d$ the number of children of the root $\rho_T$ of $T$ and conditionally on $d$, let $t_1, \dots, t_d$ be the subtrees branching from the children of the root, listed from left to right, and let $\ell_1, \dots, \ell_d$ be the labels of their roots. Conditionally on these labels, the subtrees $t_i$ are independent, and distributed according to $\Pi^{[0]}_{\ell_i}$. Let $n_{t_i}$ denote the number of leaves labelled $0$ in $t_i$, and observe that $n_t = \sum_i n_{t_i}$. Moreover:
	\begin{itemize}
		\item[\textbullet] If $\ell_i = 0$, we have $n_{t_i} = 1$.
		\item[\textbullet] Conditionally on $\{\ell_i = 1\}$, the subtree $t_i$ has distribution $\Pi^+$, and the generating function of $n_{t_i}$ is $g_{\nu_+}$.
		\item[\textbullet] Conditionally on $\{\ell_i = 2\}$, $n_{t_i}$ has the same law as the variable $\check N_2$ considered in Lemma \ref{lemmanu}, and this implies that the generating function of $n_{t_i}$ is $g_{\nu_+} \circ g_{\nu_+}$.
	\end{itemize}
	All in all, conditionally on $d$ and the labels $\ell_i$, the generating function of $n_{t_i}$ is $g_{\nu_+}^{\circ \ell_i}$, where for any measurable function $G: [0,1] \to [0,1]$ and integer $k \geq 0$, we define the $k$-th iterate $G^{\circ k} = G \circ \cdots \circ G$ ($k$ times), with the convention $G^{\circ 0}(z) = z$. It follows that $G = g_\nu$ satisfies the functional equation:
	\begin{align}
		\label{ReWritEquA}
		G(z) = \mathbb{E}_+ \left[ \prod_{\epsilon \in [\rho_t]_t} G^{\circ (1 + \epsilon)}(z) \right],
	\end{align}
	where $\mathbb E_+$ is the expectation when $t$ is sampled according to $\Pi_+$ and 
	where the product is over all \emph{components} $\epsilon$ of the vector $[\rho_t]_t$ of label increments along the edges from the root $\rho_t$ to its children (recall the definition of $[\rho_t]_t$ from section $2.3$). We can rewrite more explicitly,
	\begin{align}
		\label{EquA}G(z)= \xi(0) +\sum_{d=1}^\infty\xi(d)\sum_{i_1, \cdots, i_d \in \{-1, 0, 1\}^d}\eta^{(d)}(i_1, \cdots, i_d)\prod_{k=1}^dG^{\circ (i_k+1)}(z),
	\end{align}
	since the law of $[\rho_t]_t\in \sqcup_{d\geq 0} \{-1, 0, 1\}^d$ is given by $\Pi^+(\{[\rho_t]_t= (i_j)_{k\leq d}\})=\xi(d)\eta^{(d)}(i_1, \cdots, i_d)$.
	\begin{lemma}
		\label{UnicityFunc}
		The function $g_{\nu_+}: [0, 1] \to [0, 1]$ is the unique measurable function $G: [0, 1] \to [0, 1]$ that satisfies the functional equation (\ref{EquA}) on $[0, 1]$.
	\end{lemma}
	
	\begin{proof}
		Let $G: [0, 1] \to [0, 1]$ be a solution to (\ref{EquA}).
		For a tree $t \in \mathcal{E}_+$ and a vertex $v \in V(t)$, let $h_t(v)$ denote the height of $v$ in $t$. For any $N \geq 1$ and $t \in \mathcal{E}_+$, define
		\[
		n^N_t := \text{Card}\left\{ v \in V(t) \,\middle|\, \ell_t(v) = 0,\ h_t(v) \leq N \right\},
		\]
		the number of vertices labelled $0$ at height at most $N$. Also define the set of vertices at height $N$ with positive labels by $
		\partial_N t := \left\{ v \in V(t) \,\middle|\, h_t(v) = N,\ \ell_t(v) > 0 \right\}.$
		Observe that $n_t^1$ counts the number of children of the root with label $0$.  Fix $z \in [0, 1]$, it follows from \eqref{EquA} (or equivalently \eqref{ReWritEquA}) that
		\[
		G(z) = \mathbb{E}_+ \left[ z^{n_t^1} \prod_{v \in \partial_1 t} G^{\circ \ell_t(v)}(z) \right].
		\]
	More generally, for every $k\geq 1$, set
		\begin{align*}
			A_k:=\mathbb{E}_+\left[z^{n_t^k} \prod_{u \in \partial_k t} G^{\circ \ell_t(u)}(z) \right],
		\end{align*} so that $G(z)=A_1$. 
		For each $j \geq 1$, applying (\ref{EquA}) with $\tilde{z} = G^{\circ (j-1)}(z)$ in place of $z$, yields
		\begin{align}
			\label{TransFunc}
			G^{\circ j}(z) = \mathbb{E}^\epsilon \left[ \prod_{i=1}^{|\epsilon|} G^{\circ (j + \epsilon_i)}(z) \right],
		\end{align}
			where the expectation is taken over the distribution of the random vector $\epsilon=(\epsilon_1, \cdots, \epsilon_{|\epsilon|})$ taking values in $\bigsqcup_{d \geq 0} \{-1, 0, 1\}^d$ which is distributed under $\mathbb P^\epsilon$ as the vector of label increments of children of the root. For $u\in \partial_kt$, we have using the identity (\ref{TransFunc}) with $j=\ell_t(u)$ that:
		\begin{align*}
			G^{\circ \ell_t(u)}(z)=\mathbb{E}^\epsilon\left[\prod_{i=1}^{|\epsilon|} G^{\circ (\ell+ \epsilon_i)}(z)\right]_{\ell=\ell_t(u)},
		\end{align*}
	It follows that, for every $k\geq 1$:
		\begin{align*}
		A_k
			&= \mathbb{E}_+\left[z^{n_t^k} \prod_{u \in \partial_k t} \mathbb{E}^\epsilon\left[\prod_{i=1}^{|\epsilon|} G^{\circ (\ell+ \epsilon_i)}(z)\right]_{\ell=\ell_t(u)} \right].
		\end{align*}
		Given $\ell_t(u)$, the vector $(\ell_t(u) + \epsilon_1, \dots, \ell_t(u) + \epsilon_{|\epsilon|})$ has the same distribution as the labels of the children of $u$ in the tree $t$. Furthermore, the collection of all children of vertices in $\partial_k t$ is precisely the set $
		\left\{ v \in V(t) \,\middle|\, h_t(v) = k+1 \right\}$.
		It follows that:
		\begin{align*}
			A_k = \mathbb{E}_+\left[z^{n_t^k} \prod_{\substack{v \in V(t) \\ h_t(v) = k+1}} G^{\circ \ell_t(v)}(z)\right] 
			= \mathbb{E}_+\left[z^{n_t^{k+1}} \prod_{v \in \partial_{k+1} t} G^{\circ \ell_t(v)}(z)\right]
			= A_{k+1}.
		\end{align*}
		The second equality holds because the set $\{ v \in V(t) \mid h_t(v) = k+1 \}$ is the disjoint union of $\partial_{k+1} t$ and the $n_t^{k+1} - n_t^k$ vertices at height $k+1$ labelled $0$.	By induction on $k$, we therefore obtain:
		\[
		G(z) =A_1=A_k= \mathbb{E}_+\left[z^{n_t^k} \prod_{v \in \partial_k t} G^{\circ \ell_t(v)}(z)\right], \quad \text{for every } k \geq 1.
		\]
		Now, observe that for any $t \in \mathcal{E}_+$, we have $n_t^k = n_t$ and $\partial_k t = \emptyset$ for all $k$ greater than the height of $t$. In particular, since both $G$ and $z$ are bounded by $1$, we may apply the dominated convergence theorem to conclude:
		\[
		\mathbb{E}_+\left[z^{n_t^k} \prod_{v \in \partial_k t} G^{\circ \ell_t(v)}(z)\right] \xrightarrow[k \to \infty]{} \mathbb{E}_+[z^{n_t}] = g_{\nu_+}(z).
		\]	
		Hence, we deduce that $G(z) = g_{\nu_+}(z)$ for all $z \in [0,1]$, completing the proof of the lemma.
	\end{proof}
	
	\noindent There are some natural models for which we manage to find an explicit expression for $g_\nu$. We record them in the following Proposition. In all cases considered below, the displacement distributions $\eta^{(d)}$ are symmetric (meaning that if $X\sim \eta^{(d)}$ then $-X\sim \eta^{(d)}$ as well). In particular, this implies that $\nu_+= \nu_-$ and we write simply $g_\nu$ to designate $g_{\nu_+}$ and $g_{\nu_-}$.
	\begin{proposition}[A few explicit generating functions]
		\label{ExplicitForm} \; 
		
		\begin{enumerate}
			\item If $\xi$ is critical geometric (i.e. $\xi(k)=2^{-k-1}$) and label displacements along edges are uniform in $\{-1, 1\}$  and independent (i.e. $\eta^{(k)}= \mathcal U(\{-1, 1\})^{\otimes k}$, for every $k\geq 1$), then:
			\begin{align}
				\label{g1}
				g_{\nu_1}(z)=\frac{-2 z^2+9z-4 + 2\sqrt{4-z}(1 - z)^{3/2}}{z(4-z)}.
			\end{align}
			\item If $\xi$ is critical geometric and label displacements along edges are uniform in $\{-1,0, 1\}$ and independent (i.e. $\eta^{(k)}= \mathcal U(\{-1, 0, 1\})^{\otimes k}$, for every $k\geq 1$), then
			\begin{align}
				\label{g2}
				g_{\nu_2}(z)=\frac{-z^2+6z-3+\sqrt{9-z}(1-z)^{3/2}}{2z}.
			\end{align}
			\item In the case of "incomplete binary trees with natural embedding", i.e. when $\xi= \frac 1 4 \delta_{0} + \frac 1 2 \delta_1+\frac 1 4 \delta_2$, and $\eta^{(1)}=\frac 12(\delta_{-1}+\delta_1)$, $\eta^{(2)}= \delta_{\{-1, 1\}}$, then
			\begin{align}
				\label{g3}
				g_{\nu_3}(z)=\frac{-  z^2+ 16 z - 3 + \sqrt{49-z} (1 - z)^{3/2}}{2(z+5)}.
			\end{align}
			\item In the case of "complete binary trees with natural embedding" i.e when $\xi= \frac 12 (\delta_0+\delta_2)$ and where $\eta^{(2)}=\delta_{\{-1, 1\}}$:
			\begin{align}
				\label{g4}
				g_{\nu_4}(z)= \frac{-z^2+10z-3+\sqrt{25-z}(1-z)^{3/2}}{2(z+2)}.
			\end{align}
		\end{enumerate}	
	\end{proposition}
	\noindent	It is interesting to see how similar the expressions for $g_\nu$ are for these models. In particular, all these generating functions are  degree $2$ algebraic (i.e. solutions of a degree $2$ polynomial equations whose coefficients are polynomials in $z$). This might hint at the fact that there are general formulas for $g_\nu$ for larger classes of models but we leave this extension to the reader.
	
	\begin{remark}
		The second model is related to random quadrangulations, and the expression (\ref{g2}) for $g_{\nu_2}$ already appeared in Theorem 2 of \cite{Krikun} in connection with the exploration processes on the UIPQ (Uniform Infinite Planar Quadrangulation). The third and forth cases correspond to free uniform binary and complete binary trees with "natural" embedding (i.e. label increments are $-1$ for left child and $+1$ for right child). The third model (or more precisely a version conditioned on the number of vertices) is considered in \cite{chapuy2022notedensityiserelated}. The expression of the generating function for the fourth model was already found by Bousquet--Mélou \cite{BMpersonalcommunication}. Interestingly, the generating functions above are all in the domain of attraction of the same totally asymmetric $3/2$-stable law.  In fact, in the $4$ cases the expansion near the main singularity $z=1$ gives:
		\begin{align*}
			g_{\nu}(z)\underset{z\to 1}{=} 1-(1-z)+ \sqrt{\frac 23} \frac{\sigma_{\xi}}{\sigma_\eta} (1-z)^{3/2}+O((1-z)^2).
		\end{align*}
		We believe that such an expansion holds very generally whenever $\nu$ is centred and $\xi$ is critical, under an appropriate moment condition. In the companion paper \cite{Companion}, we actually prove a related result for all critical Galton--Watson trees (under an exponential moment condition) with i.i.d. displacements along edges.
	\end{remark}
	\begin{proof}[Proof of Proposition \ref{ExplicitForm}]
		For the first case, we have $\xi(k)=2^{-k-1}$ and $\eta^{(d)}(i_1, \cdots, i_d)= 2^{-d}$ for every $(i_1, \cdots, i_d) \in \{-1, 1\}^d$. In this case (\ref{EquA}) gives:
		\begin{align*}
			g_{\nu_1}(z)&= \sum_{d=0}^\infty2^{-d-1}\sum_{i_1, \cdots, i_d \in \{-1, 1\}^d}\frac{1}{2^d}\prod_{k=1}^d(\mathds 1_{i_k=1}z+ \mathds 1_{i_k=-1}g_{\nu_1}(g_{\nu_1}(z)))\\
			&= \sum_{d=0}^\infty2^{-d-1}\left(\frac{1}{2}(z+g_{\nu_1}(g_{\nu_1}(z)))\right)^d\\
			&=\frac{2}{4-z-g_{\nu_1}(g_{\nu_1}(z)))}.
		\end{align*}
		By a similar argument we get in the second case:
		\begin{align*}
			g_{\nu_2}(z)=\frac{3}{6-z-g_{\nu_2}(z)-g_{\nu_2}(g_{\nu_2}(z))}.
		\end{align*}
		
		\noindent For the third case, the right hand side of (\ref{EquA}) only has $4$ nonzero terms, and
		\begin{align*}
			g_{\nu_3}(z)&= \frac{1}{4}\left(1+ z+ g_{\nu_3}(g_{\nu_3}(z))+ zg_{\nu_3}(g_{\nu_3}(z))\right) = \frac{1}{4}(1+z)(1+g_{\nu_3}(g_{\nu_3}(z))).
		\end{align*}
		In the fourth case, only two terms are nonzero and (\ref{EquA}) gives $	g_{\nu_4}(z)= \left(1+zg_{\nu_4}(g_{\nu_4}(z))\right)/2$.
		 It is straightforward to check that the right-hand sides of (\ref{g1}), (\ref{g2}), (\ref{g3}) and (\ref{g4}) provide the respective solutions of these functional equations. Lemma \ref{UnicityFunc} completes the proof.
	\end{proof}
	
	\section{A Markov property for discrete local times}
	\subsection{The case of unconditioned trees}
	Let $T$ be a random labelled plane tree sampled according to the law $\Pi_0$.
	Recall that for $m\geq 1$, the variable $X_m^+$ (resp. $X_m^-$) is the number of oriented edges in $T$ going from a vertex with label $m-1$ to a vertex with label $m$ (resp. from a vertex with label $m$ to a vertex with label $m-1$). Equivalently, $X_m^+$ is the number of vertices at even height in $\tilde {\mathcal F}_T^m$, and $X_m^-$ is the number of vertices at odd height in $\tilde {\mathcal F}_T^m$. Remarkably, the two-dimensional process $(X_{m}^+, X_{m}^-)_{m\geq 1}$ is a Markov chain in this general setting, and the goal of this section is to prove this property.
	We begin with a combinatorial lemma that will prove useful in the proof. Here, a \emph{bicoloured plane forest} is a finite ordered collection of rooted plane trees in which each vertex is assigned a sign, either $+$ or $-$ and the signs alternate along ancestral lines. Vertices assigned with the sign $+$ are called positive vertices, and similarly for negative vertices.
	\begin{lemma}
		\label{lemmaCounting}
		Let $n\geq 1$, $p\geq n$ and $q\geq 0$ be integers, and fix nonnegative integers  $n^+_1, n^+_2, \cdots n^+_p$ and $n^-_1, \cdots, n^-_q$ such that:
		\begin{align}
			\label{relapq}
			p= n+ \sum_{i\leq q }n^-_i,  & \ \ \ \ \ \ \ \ \  \ \ \  q= \sum_{j\leq p} n^+_j.
		\end{align}
		Let $\mathcal S(n, (n_j^+)_{j\leq p}, (n_i^-)_{i\leq q})$ be the set of all bicoloured plane forests with $n$ trees rooted at positive vertices, with vertex set $\{ u_1, \cdots , u_q, v_1, \cdots, v_p\}$, such that:
		\begin{itemize}
			\item[a.] the negative vertices (i.e., the vertices at odd height in the forest) are $u_1, \cdots, u_q$ and have respectively $n^-_1, \cdots, n^-_q$ children,
			\item[b.] the positive vertices (i.e., the vertices at even height in the forest) are $v_1, \cdots, v_p$ and have respectively $n^+_1, \cdots, n^+_p$ children.
		\end{itemize}
		Then, the cardinal of $\mathcal S(n, (n_j^+)_{j\leq p}, (n_i^-)_{i\leq q})$ is $q!(p-1)! n$.
	\end{lemma}
	\begin{proof}
		First, note that by construction, the children of the vertices \( v_1, \dots, v_p \) are exactly the vertices \( u_1, \dots, u_q \).  
		To construct a forest in \( \mathcal{S}(n, (n_j^+)_{j\leq p}, (n_i^-)_{i\leq q}) \), we begin by choosing the children of \( v_1, \dots, v_p \) in left-to-right order. This is equivalent to choosing one of the \( q! \) permutations of the sequence \( (u_1, \dots, u_q) \).
		
		For each \( i \in \{1, \dots, p\} \), let \( A_i \subset \{1, \dots, q\} \) be the set of all indices \( j \) such that \( u_j \) is a child of \( v_i \). Now, consider collapsing each vertex \( v_i \) together with its children into a single vertex \( w_i \), and let  
		\[
		\ell_i = \sum_{j \in A_i} n_j^-,
		\]  
		so that \( w_i \) has exactly \( \ell_i \) children. Choosing an element of \( \mathcal{S}(n, (n_j^+)_{j\leq p}, (n_i^-)_{i\leq q}) \) now amounts to choosing a plane forest with \( n \) trees and vertex set \( \{w_1, \dots, w_p\} \), where each \( w_i \) has \( \ell_i \) children. Write $\mathcal F(\ell_1, \cdots, \ell_p)$ for the set of all such forests. For every \( \ell \geq 0 \), let  
		\[
		r_\ell = \text{Card}\{ i \in \{1, \dots, p\} : \ell_i = \ell \},
		\]  
		and let \( \ell_\star = \max\{ \ell \geq 0 : r_\ell > 0 \} \). By Theorem~5.3.10 of \cite{Stanley}, the number of plane forests with \( n \) trees and \( p \) vertices such that exactly \( r_\ell \) vertices have \( \ell \) children (for every \( \ell \)) is:
		\[
		\frac{n}{p} \binom{p}{r_0, r_1, \dots, r_{\ell_\star}}.
		\]
		From such a forest we can obtain an element of $\mathcal F(\ell_1, \cdots, \ell_p)$ by calling $w_j$ a vertex with $\ell_j$ children. For every value of $\ell$ there are $r_\ell!$ ways of choosing the vertices $w_j$ such that $\ell_j=\ell$. Hence:
		\[
	\text{Card}\ \mathcal F(\ell_1, \cdots, \ell_r)=\frac{n}{p} \binom{p}{r_0, r_1, \dots, r_{\ell_\star}} r_0! \cdots r_{\ell_\star}! = n (p-1)!.
		\]
		To obtain the cardinality of \( \mathcal{S}(n, (n_j^+)_{j\leq p}, (n_i^-)_{i\leq q}) \), we multiply this result by \( q! \), accounting for the initial choice of the permutation of \( (u_1, \dots, u_q) \). We conclude that $$
		\text{Card}\ \mathcal{S}(n, (n_j^+)_{j\leq p}, (n_i^-)_{i\leq q}) = q! \cdot (p-1)! \cdot n,$$
		as claimed.
	\end{proof}
	
	 For every $r\geq 0$ and (positive or negative) tree excursions  $\tau_1, \tau_2, \cdots, \tau_r, t$, write
	 \begin{align*}
	 	C_t(\tau_1, \cdots, \tau_r)= \text{Card}\{i\leq r \mid \tau_i=t\}.
	 \end{align*} for the number of times $t$ appears in the sequence $\tau_1, \cdots, \tau_r$.  Fix $m\geq 1$ and consider the excursions $\{\tau_i^{m, -}\}_{i\leq X_m^-}\cup\{\tau_j^{m, +}\}_{j\leq X_m^+}$ of the tree $T$ at level $m$.
	For each $t \in \mathcal{E}_+$, define
	\[
	\mathtt{c}_t^{m, +} := C_t(\tau_1^{m, +}, \cdots, \tau_{X_m^+}^{m, +}),
	\]
	the number of occurrences of the tree excursion $t$ in the multiset of positive excursions above level $m$. Similarly, for $t \in \mathcal{E}_-$, we define
	\[
	\mathtt{c}_t^{m, -} := C_t(\tau_1^{m, -}, \cdots, \tau_{X_m^-}^{m, -}).
	\]
	The collections $(\mathtt{c}_t^{m, +})_{t \in \mathcal{E}_+}$ and $(\mathtt{c}_t^{m, -})_{t \in \mathcal{E}_-}$ fully encode the (unordered) collections $\{\tau_1^{m, +}, \dots, \tau_{X_m^+}^{m, +}\}$ and $\{\tau_1^{m, -}, \dots, \tau_{X_m^-}^{m, -}\}$, respectively. By construction we have the following identities:
	\begin{align*}
		X_m^+= \sum_{t\in \mathcal E_+}  \mathtt c_t^{m, +} = N_m+\sum_{t\in \mathcal E_-}  \mathtt c_t^{m, -} n_t, \quad \quad X_m^-=\sum_{t\in \mathcal E_-}  \mathtt c_t^{m, -}= \sum_{t\in \mathcal E_+}  \mathtt c_t^{m, +} n_t,
	\end{align*}
	where we recall that $n_t$ denotes the number of leaves with label $0$ in the excursion $t$.  Note moreover that $X_{m+1}^+$ and $X_{m+1}^-$ are measurable functions of $(\mathtt c_t^{m, +})_{t\in \mathcal E_+}$.
	
	\begin{theorem}[Markov property]
		\label{discreteMarkovProperty}
			Let $m\geq 1$, conditionally on $(X_m^+, X_m^-)$ the collection $(\mathtt{c}_t^{m, +})_{t \in \mathcal{E}_+}$ is independent of both the root component $T^{[m]}$ truncated at level $m$ and the collection $(\mathtt{c}_t^{m, -})_{t\in \mathcal{E}_-}$. Moreover, conditionally on $\{X_m^+= p, X_m^-=q\}$, we have:
		\begin{align}
			\label{LawExcu}
			(\mathtt c_t^{m, +})_{t\in \mathcal E_+} \overset{(d)}{=} (C_t(\tau_1, \cdots, \tau_p))_{t\in \mathcal E_+},
		\end{align}
		where $\{\tau_1, \cdots, \tau_p\}$ is the collection of $p$ independent tree excursions with law $\Pi^+$, which is conditioned on the event $\{n_{\tau_1}+ \cdots n_{\tau_p}=q\}$. In other words the unordered collection $\{\tau_1^{m, +}, \cdots, \tau_{X_m^+}^{m, +}\}$ of tree excursions above level $m$ is distributed as the unordered collection of $p$  independent tree excursions with law $\Pi^+$ conditioned on their total number of leaves labelled $0$ being $q$.
		
		 In particular, the process $(X_{k}^+, X_{k}^-)_{k \geq 1}$ is a time-homogeneous Markov process.
	\end{theorem}
	
	\begin{proof}[Proof of Theorem \ref{discreteMarkovProperty}]
		Fix a labelled plane tree $t_\rho\in \mathbb T_0$ with exactly $n$ leaves labelled $m$ and such that all other vertices have label smaller than or equal to $m-1$. Fix $p\in\{0, 1,2, \cdots\}$ and $q\in\{0, 1,2,\cdots\}$, and consider a collection $(c_t^+)_{t\in \mathcal E_+}$ of nonnegative integers indexed by positive tree excursions and a collection $( c_t^-)_{t\in \mathcal E_-}$ of nonnegative integers indexed by negative tree excursions such that:
		\begin{align}
			\label{coeff}
			p=\sum_{t\in \mathcal E_+}  c_t^+ = n+\sum_{t\in \mathcal E_-}  c_t^- n_t, \quad \quad q=\sum_{t\in \mathcal E_-}  c_t^-= \sum_{t\in \mathcal E_+}  c_t^+ n_t.
		\end{align}
		In particular, the collections $( c_t^+)_{t\in \mathcal E_+}$ and $( c_t^-)_{t\in \mathcal E_+}$ must have only a finite number of nonzero terms.
		We want to evaluate the probability:
		\begin{align*}
			A_{t_\rho, (c^+_t)_t, (c_t^-)_t}= \Pi_0\Big(T^{[m]}=t_\rho, \ (\mathtt c_t^{m, +})_{t\in \mathcal E_+}= (c_t^{ +})_{t\in \mathcal E_+}, \ (\mathtt c_t^{m, -})_{t\in \mathcal E_-}= (c_t^{-})_{t\in \mathcal E_-} \Big).
		\end{align*}
		Recall the notation $\Phi(t_\rho, F, (t_v)_{v\in V(F)})$ introduced in section \ref{Defdecomp}, where $F$ is a plane forest with $n$ trees and where $(t_v)_{v\in V(F)}$ is an admissible family in $M(F)$. Computing the quantity in the last display amounts to summing the quantities $\Pi_0(\Phi(t_\rho, F, (t_v)_{v\in V(F)}))$ over all plane bicoloured forests $F$ with $n$ trees and admissible families $(t_v)_{v\in V(F)}\in M(F)$ such that for every positive (resp. negative) tree excursion $t$, the number of occurrences of $t$ in $(t_v)_{v\in V(F)}$ is exactly $c_t^+$ (resp. $c_t^-$). Write $\mathfrak F$ for the set of all such pairs $(F, (t_v)_{v\in V(F)})$. We have using (\ref{FormulaPiPhi}) that for every $(F, (t_v)_v)\in \mathfrak F$,
		\begin{align*}
			\Pi_0\left(\Phi\left(t_\rho, F, (t_v)_v\right)\right)=  \Pi_0^{[m]}(t_\rho)\prod_{v\in V(F)} \Pi^{\epsilon_v}(t)= 	 \Pi_0^{[m]}(t_\rho)\prod_{t\in \mathcal E_+} \Pi^{+}(t)^{c_t^+}\prod_{t\in \mathcal E_-} \Pi^-(t)^{c_t^-},
		\end{align*} 
		where $\epsilon_v=+$ if $v$ is at even height in $F$ and $\epsilon_v=-$ otherwise. This probability is independent of the particular choice of $(F, (t_v)_v)\in \mathfrak F$ and it follows that:
		\begin{align*}
			A_{t_\rho, (c_t^+)_t, (c_t^-)_t}= \sum_{(F, (t_v)_v)\in \mathfrak F} \Pi_0\left(\Phi\left(t_\rho, F, (t_v)_v\right)\right)
			= \text{Card}(\mathfrak F) \times \Pi_0^{[m]}(t_\rho)\prod_{t\in \mathcal E_+} \Pi^{+}(t)^{c_t^+}\prod_{t\in \mathcal E_-} \Pi^-(t)^{c_t^-}.
		\end{align*} 
		We now proceed to compute the value of $\text{Card}(\mathfrak F)$. To do so, we first distinguish between copies of the same excursion; fix a sequence $(t_1^+, \cdots, t_p^+)$ of positive tree excursions in which every $t\in \mathcal E_+$ appears exactly $c_t^+$ times. Similarly fix a sequence $(t_1^-, \cdots, t_q^-)$ of negative tree excursions in which each $t\in \mathcal E_-$ appears exactly $c_t^-$ times.
		
		Recall the definition of the set $\mathcal S(n , (n_{t_j^+})_{j\leq q}, (n_{t_i^-})_{i\leq p})$ in Lemma \ref{lemmaCounting}. Note that by construction, this set exactly corresponds to the set of bicoloured plane forests $F$ of $n$ trees on the vertex set $\{u_1, \cdots , u_q, v_1, \cdots, v_p\}$ such that setting $t_{u_i}= t^+_i$ and $t_{v_j}= t^-_j$ for every $i\leq q$ and $j\leq p$ yields an admissible family $(t_w)_{w\in \{u_i\}\cup \{v_j\}}\in M(F)$, and we have $(F, (t_w)_{{w\in \{u_i\}\cup \{v_j\}}})\in \mathfrak F$. This gives a mapping
		 \begin{align*}
		 	\mathcal S(n , (n_{t_j^+})_{j\leq q}, (n_{t_i^-})_{i\leq p}) \to \mathfrak F,
		 \end{align*}
		  that is onto but not injective since we need to account for the fact that permuting vertices $u_i$ and $u_j$ such that $t_i^+= t_j^+$ would give distinct elements of $\mathcal S(n , (n_{t_j^+})_{j\leq q}, (n_{t_i^-})_{i\leq p})$ that project to the same element in $\mathfrak F$. There are exactly $\prod_{t\in \mathcal E_-} c_t^-!$ such permutations of $\{u_1, u_2, \cdots, u_p\}$, and similarly, we can permute the $v_j$'s in $\prod_{t\in \mathcal E_+} c_t^+!$ ways and still get the same element of $\mathfrak F$. All in all, it follows that we have:
		\begin{align*}
			\text{Card}(\mathfrak F)= \frac{\text{Card} \big( \mathcal S(n , (n_{t_j^+})_{j\leq q}, (n_{t_i^-})_{i\leq p})\big)}{\prod_{t\in \mathcal E_+} c_t^+! \prod_{t\in \mathcal E_-} c_t^-!}= \frac{n \cdot (p-1)! \cdot q!}{\prod_{t\in \mathcal E_+} c_t^+! \prod_{t\in \mathcal E_-} c_t^-!}
		\end{align*}
		So that finally:
		\begin{align}
			\label{ExpressionA}
			A_{t_\rho, (c_t^+)_t, (c_t^-)_t}=  n \times q!(p-1)! \times  \Pi_0^{[m]}(t_\rho)\times\prod_{t\in \mathcal E_-} \frac{\Pi^{-}(t)^{c_t^-}}{c_t^-!} \times \prod_{t\in \mathcal E_+} \frac{\Pi^+(t)^{c_t^+}}{c_t^+!}.
		\end{align}
		Note that $n$ is determined from $t_\rho$, and $p$ and $q$ are determined from $n$ and $(c_t^-)_{t\in \mathcal E_-}$ using (\ref{coeff}). In particular, summing over all choices of the collection $(c_t^+)_{t\in \mathcal E_+}$ we get:
		\begin{align}
			\label{ExpressionA2}
			\Pi_0\left(\tau_\rho^m=t_\rho, \ (\mathtt c_t^{m, -})_{t\in \mathcal E_-}= (c_t^{-})_{t\in \mathcal E_-}\right)=  n\times  q!(p-1)! \times  \Pi_0^{[m]}(t_\rho)\prod_{t\in \mathcal E_-} \frac{\Pi^{-}(t)^{c_t^-}}{c_t^-!} \times Z_{p, q},
		\end{align}
		where $Z_{p, q}$ is the normalizing constant:
		\begin{align}
			\label{ExpressionZpq}
			Z_{p, q}= \sum_{(d_t^+)_{t\in \mathcal E_+}} \prod_{t\in \mathcal E_+} \frac{\Pi^+(t)^{d_t^+}}{d_t^+!},
		\end{align}
		the sum being over all families $(d_t^+)_{t\in \mathcal E_+}$ satisfying $q=\sum_{t} d_t^+n_t$ and $p= \sum_{t} d_t^+$. Using (\ref{ExpressionA}) and (\ref{ExpressionA2}) this finally implies that:
		\begin{align}
			\label{conditionalprobability}
			\mathbb P\Big((\mathtt c_t^{m, +})_{t\in \mathcal E_+}=(c_t^+)_{t\in \mathcal E_+} \ \Big| \ T^{[m]}=t_\rho, \ (\mathtt c_t^{m, -})_{t\in \mathcal E_-}= (c_t^-)_{t\in \mathcal E_-} \Big)= \frac{1}{Z_{p, q}}\prod_{t\in \mathcal E_+} \frac{\Pi^+(t)^{c_t^+}}{c_t^+!}.
		\end{align}
		Note that on the event $\{T^{[m]}=t_\rho, (\mathtt c_t^{m, -})_{t\in \mathcal E_-}= (c_t^-)_{t\in \mathcal E_-}\}$, we have $X_m^+=p$ and $X_m^-=q$. 
		Since the conditional probability (\ref{conditionalprobability}) only depends on $p$, $q$ and $(c_t^+)_{t\in \mathcal E_+}$, this proves that $(X_m^+, X_m^-)_{m\geq 1}$ is a time-homogeneous Markov chain.\\
		
		We now proceed to show \eqref{LawExcu}. Let $\tau_1, \cdots, \tau_p$ be independent tree excursions with law $\Pi^+$. Fix a collection of nonnegative integers $(d_t)_{t\in \mathcal E_+}$ such that $\sum_{t\in \mathcal E_+} d_t=p$. Write to simplify notation $[p]=\{1, \cdots, p\}$. We have:
		\begin{align*}
			\mathbb P\left(\forall t\in \mathcal E_+, C_t(\tau_1, \cdots, \tau_p)= d_t\right)&= \sum_{\substack{A_t\subset [p], t\in \mathcal E_+\\
					(A_t)_t \text{ disjoints }; \sqcup_t A_t=[p] \\
					\text{Card} A_t= d_t  }} \mathbb P\left(\forall t\in \mathcal E_t, \forall i\in A_t, \tau_i=t\right)\\
			&=\sum_{\substack{A_t\subset [p], t\in \mathcal E_+\\
					(A_t)_t \text{ disjoints }; \sqcup_t A_t=[p] \\
					\text{Card} A_t= d_t  }} \prod_{t\in \mathcal E_+} \Pi^+(t)^{d_t}.
		\end{align*}
		It follows that:
		\begin{align}
			\label{ExprIndeExcu}
			\mathbb P\left(\forall t\in \mathcal E_+, C_t(\tau_1, \cdots, \tau_p)= d_t\right)=\frac{p!}{\prod_{t\in \mathcal E_+}d_t!}\prod_{t\in \mathcal E_+}\Pi^+(t)^{d_t}.
		\end{align}
		In particular, recall the definition (\ref{ExpressionZpq}) of the coefficient $Z_{p, q}$. According to the last computation:
		\begin{align*}
			Z_{p, q}&= \sum_{\substack{(d_t)_{t\in \mathcal E_+}\\ \sum_t d_t=p, \ \ \sum_t d_t n_t= q}} \frac{1}{p!}	\mathbb P\left(\forall t\in \mathcal E_+, C_t(\tau_1, \cdots, \tau_p)= d_t\right)\\
			&= \frac{1}{p!}\mathbb P\left(\sum_{t\in \mathcal E_+}n_t C_t(\tau_1, \cdots, \tau_p)=q\right).
		\end{align*}
		Hence we can identify:
		\begin{align}
			\label{ExprZPQ}
			Z_{p, q}=\frac{1}{p!}\mathbb P\left(n_{\tau_1}+\cdots+ n_{\tau_p}=q\right).
		\end{align}
		Write $\mathbb P_{p,q}$ for the law of $(\mathtt c_t^{m, +})_{t\in \mathcal E_+}$, conditionally on $X_m^+=p$ and $X_m^-=q$. It follows from (\ref{conditionalprobability}) and formulas (\ref{ExprIndeExcu}) and (\ref{ExprZPQ}) that for all choices of nonnegative integers $(c_t)_{t\in \mathcal E_+}$ such that $\sum_t c_t=p$ and $\sum_t c_t n_t=q$:
		\begin{align*}
			\mathbb P_{p, q}((c_t)_{t\in \mathcal E_+})&=\frac{1}{Z_{p, q}}\prod_{t\in \mathcal E_+} \frac{\Pi^+(t)^{c_t}}{c_t!}\\
			&= \frac{\mathbb P(\forall t\in \mathcal E_+, C_t(\tau_1, \cdots, \tau_p)= c_t)}{\mathbb P\left(n_{\tau_1}+\cdots+ n_{\tau_p}=q\right)}\\
			&=\mathbb P\left(\forall t\in \mathcal E_+, C_t(\tau_1, \cdots, \tau_p)= c_t \ \mid \ n_{\tau_1}+\cdots+ n_{\tau_p}=q \right),
		\end{align*}
		which completes the proof of Theorem \ref{discreteMarkovProperty}.
	\end{proof}
		Recall that we defined $\check X_{k}^+$ and $\check X_{k}^-$ at the end of section $2.2$: $\check X_{k}^+$ is the number of oriented edges from a vertex labelled $-k$ to one labelled $-k+1$ and  $\check X_{k}^- $ is the number of oriented edges from a vertex labelled $-k+1$ to one labelled $-k$. By symmetry, the process $(\check X_{k}^+, \check X_{k}^-)_{k\geq 1}$ is also a Markov process and we note that this process is a function of $T^{[1]}$ and the collection $(\mathtt c_t^{1, -})_{t\in \mathcal E_-}$. In particular, it follows from the previous theorem that conditionally on the values taken by $(X_{1}^+, X_{1}^-)$, the two Markov processes $(X_{k}^+, X_{k}^-)_{k \geq 1}$ and $(\check X_k^+, \check X_k^-)_{k\geq 1})$ are independent. Moreover, note that if the displacement laws $(\eta^{(k)})_{k\geq 0}$ are symmetric, then $(X_k^+, X_k^-)$ and $(\check X_k^-, \check X_k^+)$ have the same distribution by symmetry. In this case $(\check X_{k}^-, \check X_{k}^+)_{k\geq 1}$ and $( X_{k}^+,  X_{k}^-)_{k\geq 1}$ have the same law (beware that the signs $+$ and $-$ are interchanged between the two processes). This is reminiscent of Theorem $1$ in \cite{SDELocalTimesI} and we record this property as a corollary.
	\begin{corollaire}
		\label{CoroMarkov}
		The two 2D-processes $(\check X_{k}^-, \check X_{k}^+)_{k\geq 1}$ and $( X_{k}^+,  X_{k}^-)_{k\geq 1}$ are time homogeneous Markov chains, which have the same law if the displacement laws $(\eta^{(k)})_{k\geq 0}$ are symmetric. Moreover, conditionally on the value taken by $(X_{1}^+, X_{1}^-)$, these two processes are independent.
	\end{corollaire}
	\subsection{Trees conditioned on the number of vertices}
	
	Another important model of random trees comes from considering Galton--Watson trees conditioned on their size. More precisely, fix an integer $V\geq 1$, such that the quantity
	\begin{align*}
		\mathcal Z_V:=\Pi_{0}^{}\left(|T|=V\right)= \sum_{t: |t|= V} \Pi_0(t),
	\end{align*}
	is positive, where $|T|$ is the total number of edges in $T$. We write $\Pi^{(V)}_0$ for the law of tree $T$ conditioned on $\{|T|=V\}$. 
	The process $(X_m^+, X_m^-)_{m\geq 1}$ is no longer a homogeneous Markov chain under $\Pi_0^{(V)}$ but we explain in this section how one can nonetheless recover a Markov chain by adding the information of the total size of the tree below (or equivalently above) level $m$.
More precisely,	write $M_m^-$ for the number of edges with both endpoints having label at most $m-1$:
	\begin{align*}
		M_m^-= \text{Card}\Big\{v\in V(T)\setminus \{\rho_T\} \ : \ \ell_T(v)\leq m-1 \text{ and } \ell_T(\pi_T(v))\leq m-1\Big\},
	\end{align*}
	and note that equivalently we can easily check that:
	\begin{align}
		\label{ExprMm}
			M_m^-= |T^{[m]}|+\sum_{t\in \mathcal E_-}\mathtt c_t^{m, -}|t|-X_m^+.
	\end{align}
	To simplify notation, we write in the following, for every $m\geq 1$:
	\begin{align*}
		\mathcal X_m= (X_m^+, X_m^-), \quad \text{ and } \quad \bar {\mathcal X}_m= (X_m^+, X_m^-, M_m^-).
	\end{align*}
	We are going to show that $(\bar {\mathcal X}_m)_m$ is a Markov chain under $\Pi^{(V)}_0$. We start by observing that it is a Markov chain under the unconditioned law $\Pi_0$.
\begin{claim}\label{claim}	The 3D process $(\bar {\mathcal X}_m)_{m\geq 0}$ is a time-homogeneous Markov chain under $\Pi_0$.
\end{claim}
\begin{proof}[Proof of the claim:] Fix $m\geq 1$ and $p\geq 1$, $q\geq0$, we know by Theorem \ref{discreteMarkovProperty} that conditionally on $\{X_m^+=p, X_m^-=q\}$, the collection $(\mathtt c_t^{m, +})_{t\in \mathcal E_+}$ is independent of $(\mathtt c_t^{m, -})_{t\in \mathcal E_-}$ and of $T^{[m]}$, and its law $\mathbb P_{p, q}$ is given in Theorem \ref{discreteMarkovProperty}.

 Note that, using \eqref{ExprMm}, $M_m^-$ is a measurable function of $(\mathtt c_t^{m, -})_{t\in \mathcal E_-}$ and $T^{[m]}$. In particular, conditionally on $\{X_m^+=p, X_m^-=q\}$, the collection $(c_t^{m, +})_{t\in \mathcal E_+}$ is independent of $M_m^-$. It follows that if we now condition on $\{X_m^+= p, \ X_m^-=q, \ M_m^-=v\}$ for some $v\geq 0$, the collection $(\mathtt c_t^{m, +})_{t\in \mathcal E_+}$ again has law $\mathbb P_{p, q}$ and is independent of $(\mathtt c_t^{m, -})_{t\in \mathcal E_-}$ and of $T^{[m]}$, and in particular of the family $(\bar{\mathcal X}_k)_{k<m}$.	Since $X_{m+1}^+$ and $X_{m+1}^-$ are measurable functions of $(\mathtt c_t^{m, +})_{t\in \mathcal E_+}$ and $M_{m+1}^-$ is a measurable function of $M_m^-$ and of $(\mathtt c_t^{m, +})_{t\in \mathcal E_+}$, this ensures that conditionally on $\{\bar{\mathcal X}_m= (p,q,v)\}$, the variable $\bar {\mathcal X}_{m+1}$ is independent of $(\bar {\mathcal X}_k)_{k<m}$ with its law depending on $(p, q, v)$ but not on $m$. This shows that $(\bar {\mathcal X}_m)_{m\geq 0}$ is a time-homogeneous Markov chain under $\Pi_0$, as claimed.
\end{proof}
	Note that $\{|T|=V\}$ is the event that $\bar {\mathcal X}$ hits the absorbing state $(0, 0, V)$:
	 \begin{align*}
	 	\{|T|=V\}= \{\bar {\mathcal X}_m= (0, 0, V) \text{ for all $m$ large enough}\}.
	 \end{align*}
	 By classical arguments (Doob's $h$-transform of Markov chains), the process $(\bar {\mathcal X_m})_m$ conditioned on the event $\{|T|=V\}$ is again a Markov chain, with modified stochastic kernel:
	 \begin{align}
	 	\label{DHTRANS}
	 	\Pi_0^{(V)}\left(\bar{\mathcal X}_{m+1}=(r, s, w) |  \bar {\mathcal X_m}=(p, q, v)\right)= \frac{H(r, s, w)}{H(p, q, v)} \Pi_0\left(\bar {\mathcal X}_{m+1}= (r, s, w) | \bar {\mathcal X}_m=(p, q, v)\right),
	 \end{align}
	 where $H(p, q, v)= \Pi_0\left(|T|=V \mid \bar{\mathcal X}_1=(p, q, v)\right)$. Equivalently, noting that we have:
	 \begin{align*}
	 	|T|= M_1^-+ X_1^++ \sum_{k=1}^{X_1^+}|\tau_k^{1, +}|,
	 \end{align*}
	 where we recall that $\tau_j^{ 1,+}, j\leq X_1^+$ are the tree excursions above level $1$, it follows that:
	 \begin{align*}
	 	H(p, q, v)&= \Pi_0\left(v+p+\sum_{k=1}^{p}|\tau_k^{1, +}|=V \ \Big| \ \bar{\mathcal X}_1=(p, q, v)\right).
	 \end{align*}
	 By combining Theorem \ref{discreteMarkovProperty} and the argument of the proof of Claim \ref{claim}, we know that the conditional distribution under $\Pi_0$ of the unordered collection $\{\tau_j^{1, +}, j\leq X_1^m\}$ knowing that $\bar {\mathcal X}_1=(p, q, v)$, is the law of $p$ independent positive tree excursions distributed according to $\Pi_+$ and conditioned to have exactly $q$ leaves labelled $0$. It follows that:
	 \begin{align}
	 	\label{HarmoFunc}
	 	H(p, q, v)&=\mathbb P\left(\sum_{k=1}^p|\tau_k|= V-v-p \ \Big| \ \sum_{k=1}^pn_{\tau_k}=q\right),
	 \end{align}
	 where the $\tau_k$ are independent tree excursions sampled according to $\Pi_+$.
	\begin{proposition}
		\label{discreteMarkoCondiProp}
		Fix $V\geq 1$, under $\Pi_0^{(V)}$, the 3D process $(X_m^+, X_m^-, M_m^-)_m$, 
	is a Markov chain. Its transition kernel is a $h$-transform of the transition kernel under $\Pi_0$ in the sense of \eqref{DHTRANS}, with the harmonic function $H$ given in \eqref{HarmoFunc}.
	\end{proposition}

	\begin{remark}
		\begin{enumerate}
			\item Consider symmetrically, for every $m\geq 1$ the number $\check M_m^+$ of edges with both endpoints having label greater than or equal to $-m+1$. By symmetry, $(\check X^+_m, \check X_m^-, \check M_m^+)_{m\geq 1}$ is also a Markov chain under $\Pi^{(V)}$. This independence property is analogous to the one stated in Corollary \ref{CoroMarkov} and the proof is left to the reader. Then, conditionally on the value taken by $(X^+_1, X^-_1, M_1^-)$, the 3D processes $( X^+_{k}, X^-_{k}, M_{k}^-)_{k\geq 1}$ and $(\check X^+_k, \check X^-_k, \check M_k^+)_{k\geq 1}$ are independent time-homogeneous Markov chains under $\Pi^{(V)}$.
			\item 
			Fix $m\geq 1$ and $p, q, v\geq 0$ such that $
			\Pi_0^{(V)}((X_m^+, X_m^-, M_m^-)=(p, q, v))>0$. Under $\Pi^{(V)}$ and conditionally on $\{(X_m^+, X_m^-, M_m^-)=(p, q, v)\}$, it follows from the previous discussion that the law of the unordered collection of the tree excursions above level $m$ are distributed as the collection of $p$ independent tree excursions $\tau_1, \cdots, \tau_p$ sampled according to $\Pi_+$ and conditioned on :
			\begin{align*}
				\{n_{\tau_1}+\cdots+n_{\tau_p}=q, \ \ |\tau_1|+\cdots+|\tau_p|= V-v-p\}.
			\end{align*}
		\end{enumerate}
\end{remark}
	\section{Explicit kernels for the binary tree model}
	
	We now consider the specific model of random labelled tree where each vertex has a left and right child independently with probability $1/2$. This correspond to the case $3$ of Proposition \ref{ExplicitForm}, where the underlying Galton--Watson tree has critical offspring distribution:
	\begin{align*}
		\xi= \frac{1}{4} \delta_0+\frac{1}{2} \delta_1 + \frac{1}{4} \delta_2,
	\end{align*}
	and the displacement distributions are given by $\eta^{(1)}= \frac{1}{2}(\delta_{-1}+\delta_{1})$ and $\eta^{(2)}= \delta_{(-1, 1)}$. The measure $\Pi_0$ is a probability measure on the set $\mathbb B$ of all \emph{incomplete binary trees}, where an incomplete binary tree is a rooted plane tree in which each vertex has a (possibly empty) left subtree and a (possibly empty) right subtree. Binary trees are labelled in such a way that the root has label $0$ and a left child has label one less than his parent and a right child has label one more than his parent.  In this model, each vertex has a left child with probability $1/2$ and independently a right child with probability $1/2$. If $T\in \mathbb B$, we have:
	\begin{align*}
		\Pi_0(T)= 4^{-|T|-1},
	\end{align*}
	where we recall that $|T|$ is the number of edges in $T$. We want to find a formula for the transition kernel of the Markov chain $(X^+_m, X^-_m)_{m\geq 1}$ under $\Pi_0$. Note that by construction if $X_m^+=0$ then we must have $X_m^-=0$.
	In the following, we will consider $(\mathcal X_m)_m$ as a process taking values in the state space:
	\begin{align}
	\label{statespace}
	\mathfrak S= (\mathbb N_{>0}\times \mathbb N_{\geq 0})\cup \{(0, 0)\}.
	\end{align}
	For every states $(p, q)\in \mathfrak S$ and $(r, s)\in \mathfrak S$, let
	\begin{align*}
	\mathbf P_{p, q}^{r,s}= \Pi_0\left((X_{m+1}^+, X_{m+1}^-)= (r, s) \ \mid \ (X_m^+, X_m^-)=(p, q) \right).
	\end{align*}
	 We now state the main result of this section, finding an explicit formula for the kernel $\mathbf P$. This formula will involve the probability distributions of a random walk with i.i.d. steps of law $\nu_+$, where we recall that $\nu_+$ was defined in \eqref{Numoins}. Note that since the displacement laws are symmetric we have $\nu_+=\nu_-$ and we simply write in the following $\nu=\nu_+$. Recall that for this specific model, the generating function of $\nu$ is explicitly given according to Proposition \ref{ExplicitForm} by:
	 \begin{align*}
	 	g_\nu(z)=\frac{- z^2+ 16 z - 3 + \sqrt{49-z} (1 - z)^{3/2}}{2(z+5)}.
	 \end{align*}
	 Let $(N_j)_j$ be i.i.d random variables with law $\nu$ and consider the random walk $(S_k)_k$ defined by $S_0=0$ and $S_k= \sum_{j=1}^kN_j$ for every $k\geq 1$, and write $f_k$ for the probability distribution of $S_k$. In other words, for every $l\geq 0$, write:
	\begin{align*}
			f_k(l)= \mathbb P(S_k= l).
	\end{align*}
	\begin{theorem}[Kernel for the free binary tree model]
		\label{ThmExplicitKernel}
		For every $(p, q), (r, s)\in \mathfrak S$, we have:
		\begin{align}
			\label{FormThmExplicitKernel}
			\mathbf P_{p, q}^{r, s}= \frac{p4^{-p-s}}{p+s}
			{p+s\choose r}{p+s\choose q}\frac{f_r(s)}{f_p(q)},
		\end{align}
		with the convention $\mathbf P_{0, 0}^{0, 0}=1$.
	\end{theorem}
	\begin{remark}
		\label{RemarkEnum}
		The existence of a closed-form expression for the transition kernel $\mathbf{P}$ is closely related to enumeration formulas for the number of binary trees with a given \emph{vertical edge profile}—that is, trees for which the number of oriented edges between each possible pair of labels is prescribed. Enumeration formulas of this type (see, e.g., Theorem~1 and Theorem~6 in \cite{BousquetChapuy}) were already instrumental in \cite{chapuy2022notedensityiserelated}.	More generally, we believe that similar enumeration results could be obtained for other models, provided one can derive an explicit expression for the transition kernel of the Markov chain $(X^+_k, X^-_k)_{k}$. Such results could yield new formulas for counting (possibly weighted) families of labelled trees with prescribed vertical edge profiles.
		
		Let us make this connection more precise in the case of the binary tree model. If $t \in \mathbb{B}$ and $k \geq 1$, write $(x_k^+(t), x_k^-(t))$ for the number of oriented edges in $t$ from label $k-1$ to $k$ and from $k$ to $k-1$, respectively. Likewise, define $(\check{x}_k^+(t), \check{x}_k^-(t))$ as the number of oriented edges from label $-k$ to $-k+1$ and from $-k+1$ to $-k$. Let $(p_k, q_k)_{k \geq 1}$ and $(\check{q}_k, \check{p}_k)_{k \geq 1}$ be elements of $\mathfrak{S}$ (in particular $q_k=0$ whenever $p_k=0$ and $\check p_k=0$ whenever $\check q_k=0$). Suppose that the sets
		\[
		\mathcal{S} = \left\{ k \geq 1 \mid p_k > 0 \right\}, \quad \check{\mathcal{S}} = \left\{ k \geq 1 \mid \check{q}_k > 0 \right\}
		\]
		are either empty or intervals of the form $\{1, \dots, m\}$ and $\{1, \dots, \check{m}\}$. We are interested in computing the cardinality of the set
		\[
		\mathcal{B} = \left\{ t \in \mathbb{B} \,\middle|\, \forall k \geq 1: (x_k^+(t), x_k^-(t)) = (p_k, q_k), \, (\check{x}_k^+(t), \check{x}_k^-(t)) = (\check{p}_k, \check{q}_k) \right\}.
		\]
		Note that for every $t \in \mathbb{B}$ we have $\Pi_0(t) = 4^{-1 - \sum_{k \geq 1} (p_k + q_k + \check{p}_k + \check{q}_k)}$, and hence:
		\begin{align}
			\label{displayEnum1}
			\Pi_0\left(\left\{\begin{array}{l}
				(X_k^+, X_k^-) = (p_k, q_k), \\
				(\check{X}_k^+, \check{X}_k^-) = (\check{p}_k, \check{q}_k)
			\end{array} \forall k \geq 1 \right\} \right)
			= \text{Card}(\mathcal{B}) \cdot 4^{-1 - \sum_{k \geq 1} (p_k + q_k + \check{p}_k + \check{q}_k)}.
		\end{align}
		Since the displacement laws are symmetric, Corollary~\ref{CoroMarkov} tells us that the processes $(X_k^+, X_k^-)_k$ and $(\check{X}_k^-, \check{X}_k^+)_k$ are independent conditionally on $(X_1^+, X_1^-)$ and share the same transition kernel $\mathbf{P}$. Therefore,
		\begin{align}
			\label{displayEnum2}
			\Pi_0\left(\left\{\begin{array}{l}
				(X_k^+, X_k^-) = (p_k, q_k), \\
				(\check{X}_k^+, \check{X}_k^-) = (\check{p}_k, \check{q}_k)
			\end{array} \forall k \geq 1 \right\} \right)
			= U_{p_1, q_1}^{\check{p}_1, \check{q}_1} \cdot \prod_{k=1}^m \mathbf{P}_{p_k, q_k}^{p_{k+1}, q_{k+1}} \cdot \prod_{k=1}^{\check{m}} \mathbf{P}_{\check{q}_k, \check{p}_k}^{\check{q}_{k+1}, \check{p}_{k+1}},
		\end{align}
		where we define:
		\[
		U_{p_1, q_1}^{\check{p}_1, \check{q}_1} = \Pi_0\left( (X_1^+, X_1^-) = (p_1, q_1), \, (\check{X}_1^+, \check{X}_1^-) = (\check{p}_1, \check{q}_1) \right).
		\]
	 Comparing (\ref{displayEnum1}) and (\ref{displayEnum2}) and using (\ref{FormThmExplicitKernel}), it follows that:
	\begin{align*}
		\text{Card}(\mathcal B)= 4^{m_0+\underset{k\geq 1}{\sum} (m_k +\check m_k)}\cdot U_{p_1, q_1}^{\check p_1, \check q_1}
		\cdot \prod_{i=1}^{\check m}\frac{ \check q_{i}{\check m_i \choose \check q_{i+1} } { \check m_i \choose \check p_{i}}}{\check m_i4^{\check m_i}}\frac{f_{\check q_{j+1}}(\check p_{j+1})}{f_{\check q_j}(\check p_j)}
		\cdot \prod_{j=1}^{m}\frac{ p_{j}{m_j \choose p_{j+1}}{m_j \choose q_{j} }}{m_j 4^{m_j}} \frac{f_{p_{j+1}}(q_{j+1})}{f_{p_j}(q_j)}.
	\end{align*} 	where we wrote $m_0 = \check{p}_1 + q_1 + 1$, $\check{m}_k = \check{p}_{k+1} + \check{q}_k$, and $m_k = p_k + q_{k+1}$. The factors of the form $f_{p_j}(q_j)$ and $f_{\check q_j}(\check p_j)$ cancel out, and since $f_{p_{m+1}}(q_{m+1})=f_{\check q_{\check m+1}}(\check p_{\check m+1})= f_0(0)=1$, this yields :
	\begin{align*}
		\text{Card}(\mathcal B)= \frac{4^{m_0}\cdot U_{p_1, q_1}^{\check p_1, \check q_1}}{f_{p_1}(q_1)f_{\check q_1}(\check p_1)}
		\cdot \prod_{i=1}^{\check m}\frac{ \check q_{i}}{\check m_i}{\check m_i \choose \check q_{i+1} } { \check m_i \choose \check p_{i}} 
		\cdot \prod_{j=1}^{m}\frac{ p_{j}}{m_j}{m_j \choose p_{j+1}}{m_j \choose q_{j} }.
	\end{align*}
		In fact, a special case of Theorem~6 in \cite{BousquetChapuy} implies that:
		\[
		\text{Card}(\mathcal{B}) = \frac{1}{m_0} \binom{m_0}{p_1} \binom{m_0}{\check{q}_1} \cdot \prod_{i=1}^{\check{m}} \frac{\check{q}_i}{\check{m}_i} \binom{\check{m}_i}{\check{q}_{i+1}} \binom{\check{m}_i}{\check{p}_i}
		\cdot \prod_{j=1}^{m} \frac{p_j}{m_j} \binom{m_j}{p_{j+1}} \binom{m_j}{q_j},
		\]
		and in particular, we deduce:
		\[
		U_{p_1, q_1}^{\check{p}_1, \check{q}_1} = \frac{4^{-1 - \check{p}_1 - q_1}}{1 + \check{p}_1 + q_1} \binom{1 + \check{p}_1 + q_1}{p_1} \binom{1 + \check{p}_1 + q_1}{\check{q}_1} f_{p_1}(q_1) f_{\check{q}_1}(\check{p}_1).
		\]
	It is possible to recover the formula in Theorem~\ref{ThmExplicitKernel} from Theorem~6 in \cite{BousquetChapuy} but this requires some work in particular to retrieve the “missing factor” $f_r(s)/f_p(q)$. For this reason, we provide a proof of Theorem~\ref{ThmExplicitKernel} independent of \cite{BousquetChapuy}, with the aim of offering a more probabilistic interpretation of the kernel $\mathbf{P}$ and shedding new light on the aforementioned enumeration result.
	\end{remark}
	Before proceeding to the proof of Theorem~\ref{ThmExplicitKernel}, we make a few preliminary observations. For every tree excursion $\tau$, write  $y^+_\tau$ the number of oriented edges going from label $1$ to label $2$, and by $y^-_\tau$ the number of oriented edges going from label $2$ to label $1$ (following the notations in Remark \ref{RemarkEnum} we have equivalently $y^-_\tau=x^-_2(\tau)$ and $y^+_\tau=x^+_2(\tau)$). Then, for every $m \geq 1$, note that we have:
	\begin{align}
		\label{Xfromy}
		X_{m+1}^+ = \sum_{k=1}^{X_m^+} y_{\tau_k^{m, +}}^+, \qquad X_{m+1}^- = \sum_{k=1}^{X_m^+} y_{\tau_k^{m, +}}^-.
	\end{align}
		\begin{figure}[h!]
		\label{Fig2}
		\centering
		\includegraphics[width=0.55\textwidth]{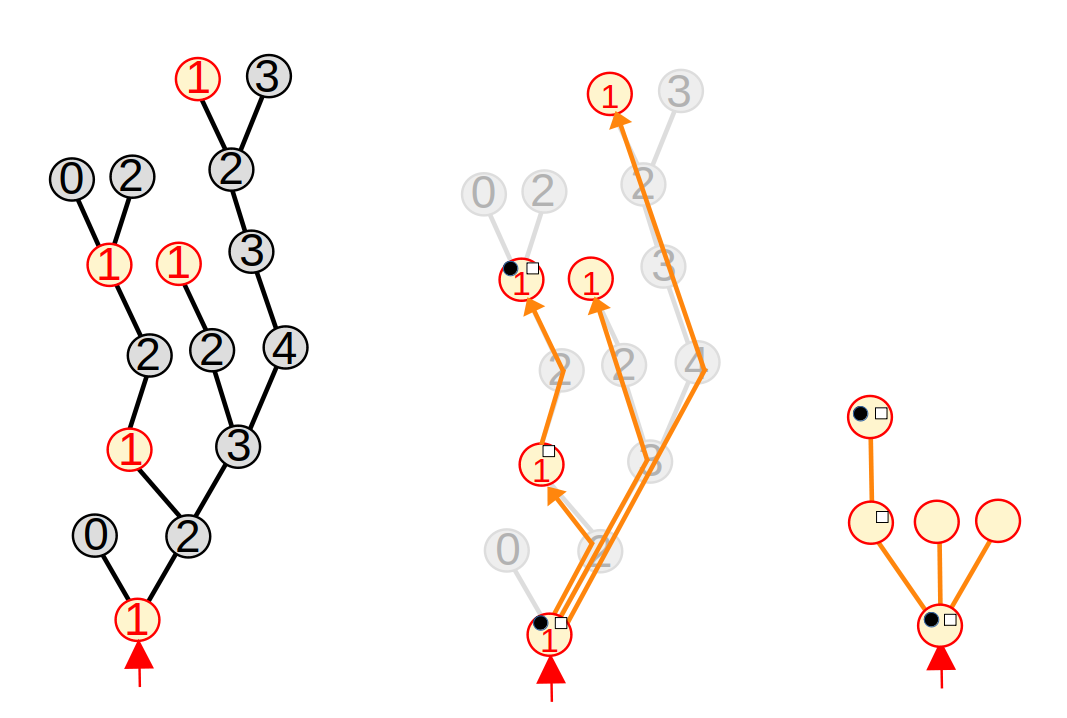}
		\caption{Consider a positive binary tree excursion $\tau$ and highlight the vertices with label $1$ (left). Connect these vertices through the geodesics in the tree $\tau$, keeping the planar structure in $\tau$ (middle). The tree $\mathfrak T_\tau$ is the tree obtained this way (right). The vertices with a right child (resp. left child) in $\tau$ are marked with a white square (resp. a black circle) and are the elements of $\mathcal R(\tau)$ (resp. $\mathcal L(\tau)$).}
	\end{figure}
	Let $\tau \in \mathcal{E}_+$ be a positive tree excursion sampled according to the law $\Pi^+$. Denote by $D$ the number of vertices labelled $1$ in $\tau$, and let $\mathcal{V} = \{u_1, \dots, u_D\}\subset V(\tau)$ be the set of all such vertices.
		We now define a tree $\mathfrak{T}_\tau$ with vertex set $\mathcal{V}$, where an edge is drawn from $u$ to $v$ if and only if $u$ is a strict ancestor of $v$ in $\tau$, and no other vertex of $\mathcal{V}$ lies strictly between $u$ and $v$ along the ancestral line. The tree $\mathfrak{T}_\tau$ inherits a natural plane tree structure from $\tau$.
	
	Next, define $\mathcal{L}(\tau) \subset \mathcal{V}$ as the set of vertices in $\mathcal{V}$ that have a left child in $\tau$, and similarly, let $\mathcal{R}(\tau) \subset \mathcal{V}$ be the set of those with a right child. Note that any vertex $v \in \mathcal{V} \setminus \mathcal{R}(\tau)$ must be a leaf in $\mathfrak{T}_\tau$, though the converse is not necessarily true. In summary, the triple $(\mathfrak{T}_\tau, \mathcal{L}(\tau), \mathcal{R}(\tau))$ defines a plane tree equipped with two distinguished subsets of vertices. This construction is illustrated in Figure~2 and will play a key role in the analysis that follows. Let $\mathfrak{P}$ denote the law of the triplet $(\mathfrak{T}_\tau, \mathcal{L}(\tau), \mathcal{R}(\tau))$, viewed as a probability distribution over the set of triples $(T, L, R)$, where $T$ is a plane tree and $L$, $R$ are two distinguished subsets of vertices in $T$.
	
	Observe that, by construction, we have:
	\begin{align}
		\label{yfromFrakT}
		|\mathfrak{T}_\tau| = y^-_\tau, \qquad |\mathcal{L}(\tau)| = n_\tau, \qquad |\mathcal{R}(\tau)| = y^+_\tau,
	\end{align}
	where $|\mathcal{L}(\tau)|$ and $|\mathcal{R}(\tau)|$ denote the cardinalities of $\mathcal{L}(\tau)$ and $\mathcal{R}(\tau)$, respectively, and we recall that $|\mathfrak{T}_\tau|$ is the number of edges in $\mathfrak{T}_\tau$ (equivalently, it is one less than the number of vertices labelled $1$ in $\tau$).
	The second and third equalities follow directly from the definitions of $\mathcal{L}(\tau)$ and $\mathcal{R}(\tau)$, noting that each vertex in the tree has at most one left and one right child. For the first equality, observe that every vertex in $\mathcal{V}$, except the root, is connected to its parent in $\tau$ by an oriented edge from label $2$ to label $1$, which contributes to $y^-_\tau$ and that we have $\text{Card}(\mathcal V)= 1+|\mathfrak T_\tau|$ since $\mathfrak T_\tau$ is a tree.\\

		We now provide a more "inductive" description of the law $\mathfrak{P}$, as inherited from the Markovian construction of the tree $\tau$ under the law $\Pi_+$. 
		Informally,	consider a vertex $v$ labelled $1$ in $\tau$, and let $\tau^{(v)}$ denote the subtree of $\tau$ induced by all vertices that are \emph{not descendants} of $v$. The subtree branching from $v$ is independent of $\tau^{(v)}$. Moreover, the vertex $v$ has a right child with probability $1/2$ (implying $v \in \mathcal{R}(\tau)$), and independently, a left child with probability $1/2$ (implying $v \in \mathcal{L}(\tau)$).	Conditionally on $v$ having a right child $u$ in $\tau$, the children of $v$ in the tree $\mathfrak{T}_\tau$ correspond to the first hitting points of label $1$ in the subtree branching from $u$ in $\tau$. In particular, the number of such children follows the law $\nu$. More rigorously, the law $\mathfrak{P}$ admits the following inductive description:\\
		
		Fix $n \geq 0$, and suppose that the set of vertices at generation $n$ has already been constructed. We construct generation $n+1$ as follows: for each vertex $u$ at generation $n$, sample two independent Bernoulli random variables $\iota_u$ and $\sigma_u$ with parameter $1/2$, independent of each other and of the past. If $\sigma_u = 1$, then $u$ gives birth to a random number $k_u$ of children at generation $n+1$, where $k_u$ is sampled according to the law $\nu$, independently of the past. Otherwise, $u$ gives birth to no child at generation $n+1$.
		
		This procedure is iterated inductively starting from a single root vertex at generation $0$. The process stops almost surely and we let $T$ be the resulting genealogical tree. We define:
		\[
		R_T = \{u \in T : \sigma_u = 1\}, \qquad L_T = \{u \in T : \iota_u = 1\}.
		\]
		Then the triplet $(T, L_T, R_T)$ is distributed according to the law $\mathfrak{P}$.
	A rigorous proof of this claim could be derived using a stopping line argument, similar to the approach used in the proof of Lemma~\ref{lemmanu} and Proposition~\ref{distributionForest}. Since such arguments are already well illustrated earlier in the paper, we omit the details here.\\
		
		Armed with this description of the law $\mathfrak{P}$, we will prove the following:
\begin{lemma}
	\label{LemmaJointlawFrakTs}
	Fix $p\geq 1$ and let $(T_i,  L_i,  R_i)_{i\leq p}$ be independent marked trees with law $\mathfrak P$. For every $q, r, s\geq 0$, we have:
	\begin{align}
		\label{JointlawFrakTs}
		\mathbb P\left(\sum_{k=1}^p| T_k|=s, \ \sum_{k=1}^p|L_k|= q, \  \sum_{k=1}^p |R_k|=r\right)=4^{-p-s}\frac{p}{p+s}{p+s\choose q}{p+s\choose r}f_r(s).
	\end{align}
\end{lemma}\noindent Before proceeding to the proof of Lemma~\ref{LemmaJointlawFrakTs}, let us first explain how it implies Theorem~\ref{ThmExplicitKernel}.
		\begin{proof}[Proof of Theorem~\ref{ThmExplicitKernel}]
			Fix $(p, q) \in \mathfrak{S}$ and $(r, s) \in \mathfrak{S}$, and suppose that $p \geq 1$ (the case $p = 0$ is trivial). Let $\tau_1, \dots, \tau_p$ be $p$ independent tree excursions sampled under the law $\Pi^+$.
			According to Theorem~\ref{discreteMarkovProperty}, conditionally on $\{(X_m^+, X_m^-) = (p, q)\}$, the law of the (unordered) collection of excursions above level $m$ is the same as the law of $\{\tau_1, \dots, \tau_p\}$ conditioned on the event $\{n_{\tau_1} + \cdots + n_{\tau_p} = q\}$. In particular, it follows from \eqref{Xfromy} that:
		\begin{align*}
			\mathbf P_{p, q}^{r,s}&= \mathbb P\left(\sum_{k=1}^py^+_{\tau_k}= r, \ \sum_{k=1}^py^-_{\tau_k}=s \ \Big|\  \sum_{k=1}^p n_{\tau_k}=q\right)\\
			&=\frac{1}{\mathbb P\left(\sum_{k=1}^pn_{\tau_k}=q\right)}\mathbb P\left(\sum_{k=1}^py^+_{\tau_k}= r, \sum_{k=1}^py^-_{\tau_k}=s, \sum_{k=1}^p n_{\tau_k}=q\right)\\
			&=\frac{1}{f_p(q)}\mathbb P\left(\sum_{k=1}^py^+_{\tau_k}= r, \sum_{k=1}^py^-_{\tau_k}=s, \sum_{k=1}^p n_{\tau_k}=q\right)
		\end{align*}	
		But according to \eqref{yfromFrakT}, we have:
		\begin{align*}
			\mathbb P\left(\sum_{k=1}^py^+_{\tau_k}= r, \ \sum_{k=1}^py^-_{\tau_k}=s, \ \sum_{k=1}^p n_{\tau_k}=q\right)&=\mathbb P\left(\sum_{k=1}^p|\mathcal R(\tau_k)|= r, \ \sum_{k=1}^p|\mathfrak T_{\tau_k}|=s, \ \sum_{k=1}^p |\mathcal L(\tau_k)|=q\right).
		\end{align*}
	The formula in Theorem~\ref{ThmExplicitKernel} now follows directly from Lemma~\ref{LemmaJointlawFrakTs}.
		\end{proof}
	It remains to prove Lemma~\ref{LemmaJointlawFrakTs}.
		\begin{proof}[Proof of Lemma~\ref{LemmaJointlawFrakTs}]
			Let $(T_i,  L_i,  R_i)_{i\geq 1}$ be independent marked trees with law $\mathfrak P$. Consider the (infinite) ordered forest $F=(T_1, T_2, \cdots)$ and write $(w_1, w_2 \cdots)$ for the enumeration of the vertices of $F$ in depth-first-order (first enumerating all vertices of $T_1$ then all vertices of $T_2$ etc.). Fix $j\geq 1$ and let $l$ be such that $w_j\in T_l$. Write $k_j$ for the number of children of $w_j$ in the tree $T_l$ and:
			\begin{align*}
			\sigma_j=\mathds 1_{w_j\in R_l}, \quad \quad  \text{ and } \quad \quad \iota_j= \mathds 1_{w_j \in L_l}.
			\end{align*}
			It follows from the description of the law $\mathfrak P$ that $(\sigma_j)_{j\geq 1}$ and $(\iota_j)_{j\geq 1}$ are independent collections of Bernoulli variables of parameter $1/2$.
			
			We consider $(S_i)_{i\geq 0}$ the so-called \emph{\L ukasiewicz} walk associated with the enumeration $(w_j)_{j\geq 1}$ of the vertices of $F$. In other words, we let $S_0=0$ and for every $i\geq 1$:
			\begin{align*}
				S_i= \sum_{j=1}^i(k_j-1).
			\end{align*}
			It follows from the description of the law $\mathfrak P$ that $S_i$ is a random walk with i.i.d. steps which are independent of $(\iota_j)_j$ and are coupled with $(\sigma_j)_j$ in such a way that conditionally on $(\sigma_j)_j$:
			\begin{itemize}
				\item[(i)] we have $S_{j}-S_{j-1}=-1$ for every $j\geq 1$ such that $\sigma_j=0$,
				\item[(ii)] the random variables $(S_j-S_{j-1}+1)_{j\geq 1, \sigma_j=1}$ are independent with law $\nu$.
			\end{itemize}
			In particular, note that for every $n\geq 1$, we have:
			\begin{align*} S_n+n= \sum_{\substack{1\leq j\leq n\\ \sigma_j=0}}\underbrace{(S_j-S_{j-1}+1)}_{=0}+\sum_{\substack{1\leq j\leq n\\ \sigma_j=1}}\underbrace{(S_j-S_{j-1}+1)}_{=k_j}= \sum_{\substack{1\leq j\leq n\\ \sigma_j=1}} k_j.
			\end{align*}
			In particular, conditionally on $\{\sum_{j=1}^n \sigma_j=r\}$ for some $r\geq 0$, the random variable $S_n+n$ is the sum of $r$ independent random variables with law $\nu$. Using this remark and the independence of $(\iota_j)_j$ from $(S_j)_j$ and $(\sigma_j)_j$, it follows that for every $n\geq 1$ and every $r, \ell, q \geq 0$, we have:
			\begin{align*}
				\mathbb P\left(S_n=\ell, \ \sum_{j=1}^n\sigma_j=r, \ \sum_{j=1}^n \iota_j=q\right)&= \mathbb P\left(S_n+n=\ell+n \ \Big| \ \sum_{j=1}^n\sigma_j=r\right)\mathbb P\left( \sum_{j=1}^n\sigma_j=r\right)\mathbb P\left(\sum_{j=1}^n\iota_j=q\right)\\
				&=f_r(\ell+n)\mathbb P\left( \sum_{j=1}^n\sigma_j=r\right)\mathbb P\left(\sum_{j=1}^n\iota_j=q\right)
			\end{align*}
			But $\sum_{j\leq n} \iota_j$ and $\sum_{j\leq n}\sigma_j$ are binomial $\mathcal B(n, \frac 12)$ random variables. We get:
			\begin{align}
				\label{ExprJoinS}
				\mathbb P\left(S_n=\ell, \sum_{j=1}^n\sigma_j=r, \sum_{j=1}^n \iota_j=q\right)&= f_r(\ell+n)\frac{1}{2^n}{n \choose r}\frac{1}{2^n}{n\choose q}.
			\end{align}
			Fix $p\geq 1$ and write $h_{-p}=\inf\{k\geq 0 \mid S_k=-p\}$ for the first hitting time of $-p$ by the {\L ukasiewicz} walk. It follows from the standard property of \L ukasiewicz walks that $h_{-p}$ is also the first time at which we complete the  exploration of the tree $T_{p}$ in the enumeration $(w_k)_k$. In other words we have:
			\begin{align*}
				h_{-p}= \sum_{k=1}^p |T_k|+p,
			\end{align*}
			since the total number of vertices in $(T_1, \cdots, T_p)$ is $\sum_{k\leq p} |T_k|+p$. Fix $q, r, s\geq 0$, we then have:
			\begin{align*}
				\mathbb P\left(\sum_{j=1}^p| T_j|=s, \ \sum_{j=1}^p|R_j|= r, \ \sum_{j=1}^p |L_j|=q\right)&=\mathbb P\left(h_{-p}=p+s, \ \sum_{j=1}^{p+s}\sigma_j= r, \ \sum_{j=1}^{p+s}\iota_j= q\right).
			\end{align*}
			We may now apply Kemperman's formula (see e.g. Section $6.1$ in \cite{PitmanCombin}) to obtain:
			\begin{align}
				\label{Kampermann}
				\mathbb P\left( h_{-p}=p+s, \ \sum_{j=1}^{p+s}\sigma_j= r, \ \sum_{j=1}^{p+s}\iota_j= q\right)= \frac{p}{p+s}\mathbb P\left(S_{p+s}=-p, \ \sum_{j=1}^{p+s}\sigma_j= r, \ \sum_{j=1}^{p+s}\iota_j= q \right).
			\end{align}
			Indeed, $(k_j, \iota_j, \sigma_j)_{j\leq p+s}$ are i.i.d. random variables, and the event $\{\sum_{j=1}^{p+s}\sigma_j= r, \sum_{j=1}^{p+s}\iota_j= q\}$ is invariant under cyclic permutations of these triples. Hence, by the classical cyclic lemma, Kemperman's formula can be applied in this case.
			Using (\ref{ExprJoinS}) with $\ell=-p$ and $n=p+s$, we get formula (\ref{JointlawFrakTs}).
		\end{proof}

\paragraph{Kernel for the conditioned probability}	Fix $V\geq 1$. The conditional probability measure $\Pi_0^{(V)}$ corresponds to the uniform distribution over the set $\mathbb B^V$ of incomplete binary trees with $V$ edges. In other words, we have  $\Pi_0^{(V)}(T)= C_{V+1}^{-1}$ for every $t\in \mathbb B^V$,
where $C_{V+1}=\frac{1}{V+2}{2V+2 \choose V+1}$ is the Catalan number counting such trees.	Recall that $M_m^-$ is the number of edges in the tree whose endpoints both have labels at most label $m-1$. The process $(X_m^+, X_m^-, M_m^-)_m$ evolves in the state space:
	\begin{align*}
		\mathfrak S^{(V)} = (\mathbb N_{> 0} \times \mathbb N_{\geq 0}\times \{0, 1, \cdots, V\})\cup\{(0, 0, V)\}.
	\end{align*}
	For $(p, q, v),(r, s, w)\in {\mathfrak S}^{(V)}$ and $m\geq 1$ such that $\Pi_0^{(V)}((X_{m}^+, X_{m}^-, M_{m}^-)=(p, q, v))>0$, let:
	\begin{align*}
	\tilde {\mathbf P}_{p, q, v}^{r,s, w}&= \Pi_0^{(V)}\left((X_{m+1}^+, X_{m+1}^-, M_{m+1}^-)=(r, s, w) \mid (X_{m}^+, X_{m}^-, M_{m}^-)=(p, q, v) \right),
\end{align*}
be the transition kernel of the Markov chain $(X^+_m, X^-_m, M_m^-)_m$ under $\Pi_0^{(V)}$.	From the general Doob's $h$-transform identity \eqref{DHTRANS} recall that we have:
	\begin{align}	
		\label{HtransfoBinary}
			\tilde {\mathbf P}_{p, q, v}^{r,s, w}
		&=\frac{H(r, s, w)}{H(p, q, v)}\Pi_0\left((X_{m+1}^+, X_{m+1}^-, M_{m+1}^-)=(r, s, w) \mid (X_{m}^+, X_{m}^-, M_{m}^-)=(p, q, v) \right),
	\end{align}
	with the function $H$ defined in \eqref{HarmoFunc}. For every $(p, q, v)\in\mathfrak S^{(V)}$ define:
	\begin{align}
		\label{tildepi}
		\tilde f_p(q, l)= \mathbb P\left(  \sum_{k=1}^pn_{\tau_k}=q, \ \sum_{k=1}^p|\tau_k|= l\right),
	\end{align}
	with $\tau_1, \cdots, \tau_p$ are i.i.d. tree excursions distributed according to $\Pi^+$. In particular, note that \begin{align}
		\label{SimpliH}
		H(p, q, v)= \tilde f_p(q, V-v-p)/f_p(q),
	\end{align} where as before $f_p(q)= \mathbb P\left(  \sum_{k=1}^pn_{\tau_k}=q\right)$. For the binary tree model, no edges connect two vertices of same label, it follows that for every $m\geq 1$:
	 \begin{align*}
	 	M_{m+1}^-=M_m^-+X_m^++X_m^-.
	 \end{align*} In particular, if $\tilde {\mathbf P}_{p, q, v}^{r,s, w}>0$ then necessarily $w=v+p+q$ and we have in this case:
	\begin{align*}
		\Pi_0&\left((X_{m+1}^+, X_{m+1}^-, M_{m+1}^-)=(r, s, v+p+q) \mid (X_{m}^+, X_{m}^-, M_{m}^-)=(p, q, v) \right)\\
		&\qquad \qquad = \Pi_0\left((X_{m+1}^+, X_{m+1}^-)=(r, s) \mid (X_{m}^+, X_{m}^-, M_m^-)=(p, q, v) \right).
	\end{align*}
	Recall from the proof of Claim \ref{claim} that the law of $(X_{m+1}^+, X_{m+1}^-)$ under $\Pi_0$, conditionally on $\{ (X_{m}^+, X_{m}^-, M_m^-)=(p, q, v)\}$ is the same as if we only condition on $\{(X_{m}^+, X_{m}^-)=(p, q)\}$. In particular:
	\begin{align*}
			\Pi_0&\left((X_{m+1}^+, X_{m+1}^-, M_{m+1}^-)=(r, s, v+p+q) \mid (X_{m}^+, X_{m}^-, M_{m}^-)=(p, q, v) \right)\\
		&\qquad \qquad= \Pi_0\left((X_{m+1}^+, X_{m+1}^-)=(r, s) \mid (X_{m}^+, X_{m}^-)=(p, q) \right)\\
		&\qquad  \qquad = \mathbf P_{p, q}^{r, s}.
	\end{align*}
	Combining (\ref{FormThmExplicitKernel}), (\ref{HtransfoBinary}) and (\ref{SimpliH}), we finally get that:
	\begin{align*}
	\tilde {\mathbf P}_{p, q, v}^{r,s, v+p+q}=\frac{\tilde f_r(s, V-v-p-q-r)/f_r(s)}{\tilde f_p(q, V-v-p)/f_p(q)}
			\frac{p4^{-p-s}}{p+s}{p+s\choose r}{p+s\choose q}\frac{ f_r(s )}{ f_p(q)},
	\end{align*}
	which simplifies to:
	\begin{theorem}[Kernel for the conditioned binary tree model]
		\label{ThmExplicitKernelCondi}
		For every $p, q, r, s, v, w$ such that $(p, q, v)\in \mathfrak S^{(V)}$ and $(r, s, w)\in \mathfrak S^{(V)}$ we have:
		\begin{align*}
			\tilde {\mathbf P}_{p, q, v}^{r,s, w}=\left\{\begin{array}{l}
				\displaystyle	\frac{p4^{-p-s}}{p+s}{p+s\choose r}{p+s\choose q}\frac{\tilde f_r(s, V-w-r )}{\tilde f_p(q, V-v-p )} \quad \text{ if $w= v+p+q$},\\ \\
				\displaystyle	0 \qquad \qquad \qquad \qquad \quad \quad \quad \quad \quad\quad \quad \quad \quad \quad \text{ otherwise},
			\end{array}\right.
		\end{align*}
		where $\tilde f$ is defined in (\ref{tildepi}).
	\end{theorem}
	\begin{remark}
	We mention that following the same line of reasoning as in Remark \ref{RemarkEnum}, we can use Theorem~6 in \cite{BousquetChapuy} to prove that the probability:
		\begin{align*}
		\tilde 	U_{p_1, q_1}^{\check p_1, \check q_1, v}=	\Pi_0^{(V)}\left((X_1^+, X_1^-)=(p_1, q_1), \
			(\check X_1^+, \check X_1^-)=(\check p_1, \check q_1), \ M_1^-= v\right),
		\end{align*}
		has the explicit expression:
	\begin{align*}
	\tilde U_{p_1, q_1}^{\check{p}_1, \check{q}_1, v} = \binom{1 + \check{p}_1 + q_1}{p_1} \binom{1 + \check{p}_1 + q_1}{\check{q}_1}\frac{\tilde f_{p_1}(q_1, V-v-p_1)\cdot \tilde f_{\check q_1}(\check p_1, v-\check q_1)\cdot 4^{ \check{p}_1 +q_1-V}}{(1 + \check{p}_1 + q_1)\cdot C_{V+1}}.
	\end{align*}
	\end{remark}

	\section{Applications to Random Maps}
	
	A \emph{planar map} is an equivalence class (under orientation-preserving homeomorphisms of the sphere) of finite graphs embedded in the sphere. Multiple edges are allowed, and the boundary of a face need not be a simple cycle.	We focus here on \emph{pointed quadrangulations}, which are planar maps $\mathbf{q}$ satisfying the following properties:
	\begin{itemize}
		\item[\textbullet] Every face of $\mathbf{q}$ has degree $4$ (equivalently, the dual graph is $4$-valent).
		\item[\textbullet] The map is \emph{rooted}—that is, a distinguished oriented edge $e$ is specified. We denote by $\mathbf{x}_0$ and $\mathbf{x}_1$ the origin and endpoint of $e$, respectively.
		\item[\textbullet] The map is \emph{pointed}—that is, it has a distinguished vertex $\mathbf{x}_\star \in V(\mathbf{q})$, which may or may not coincide with $\mathbf{x}_0$ or $\mathbf{x}_1$.
	\end{itemize}
	
	\noindent Each pointed quadrangulation is necessarily bipartite and, thanks to the root edge, has no nontrivial automorphisms. For each $n \geq 1$, we denote by $\mathcal{Q}_n^{\bullet}$ the set of all pointed quadrangulations with $n$ faces, and write
	\[
	\mathcal{Q}^\bullet = \bigsqcup_{n \geq 1} \mathcal{Q}_n^\bullet,
	\]
	for the set of all such quadrangulations. For $q \in \mathcal{Q}^\bullet$, we denote by $|q|$ the number of faces (so that $q \in \mathcal{Q}_n^\bullet$ implies $|q| = n$), and by $\Delta_q$ the graph distance on the vertex set of $q$. We will write in the following:
	\begin{equation}
		\label{dstar}
		d_\star = \max\{ \Delta_q(\mathbf{x}_0, \mathbf{x}_\star),\ \Delta_q(\mathbf{x}_1, \mathbf{x}_\star) \}.
	\end{equation}
	 For every $k\geq 1$, we define the ball of radius $k$ centred at $\mathbf x_\star$ as the planar map $B_k(\mathbf x_\star)$ induced by keeping only the edges in $\mathbf q$ between vertices whose ends are at most at distance $k$ from $\mathbf x_\star$ (for the graph distance on $\mathbf q$). Note that this definition is not standard and differs from other references. Some of the faces of $B_k(\mathbf x_\star)$ were already in $\mathbf q$ but this procedure also creates new faces which are no longer necessarily of degree $4$. Call these new faces the \emph{external faces of $B_k(\mathbf x_\star)$}.  For every $k\geq 1$ write $C_k(\mathbf q)$ for the number of external faces and $P_k(\mathbf q)$ for the sum of the degrees of all external faces of $B_k(\mathbf x_\star)$. We extend these definitions to all $k\in \mathbb Z$, by setting by convention $C_k(\mathbf q)=1$ and $P_k(\mathbf q)=0$ whenever $k\leq 0$.\\
	
	\begin{figure}
		\centering
		\begin{minipage}[b]{0.49\textwidth}
			\includegraphics[width=\textwidth]{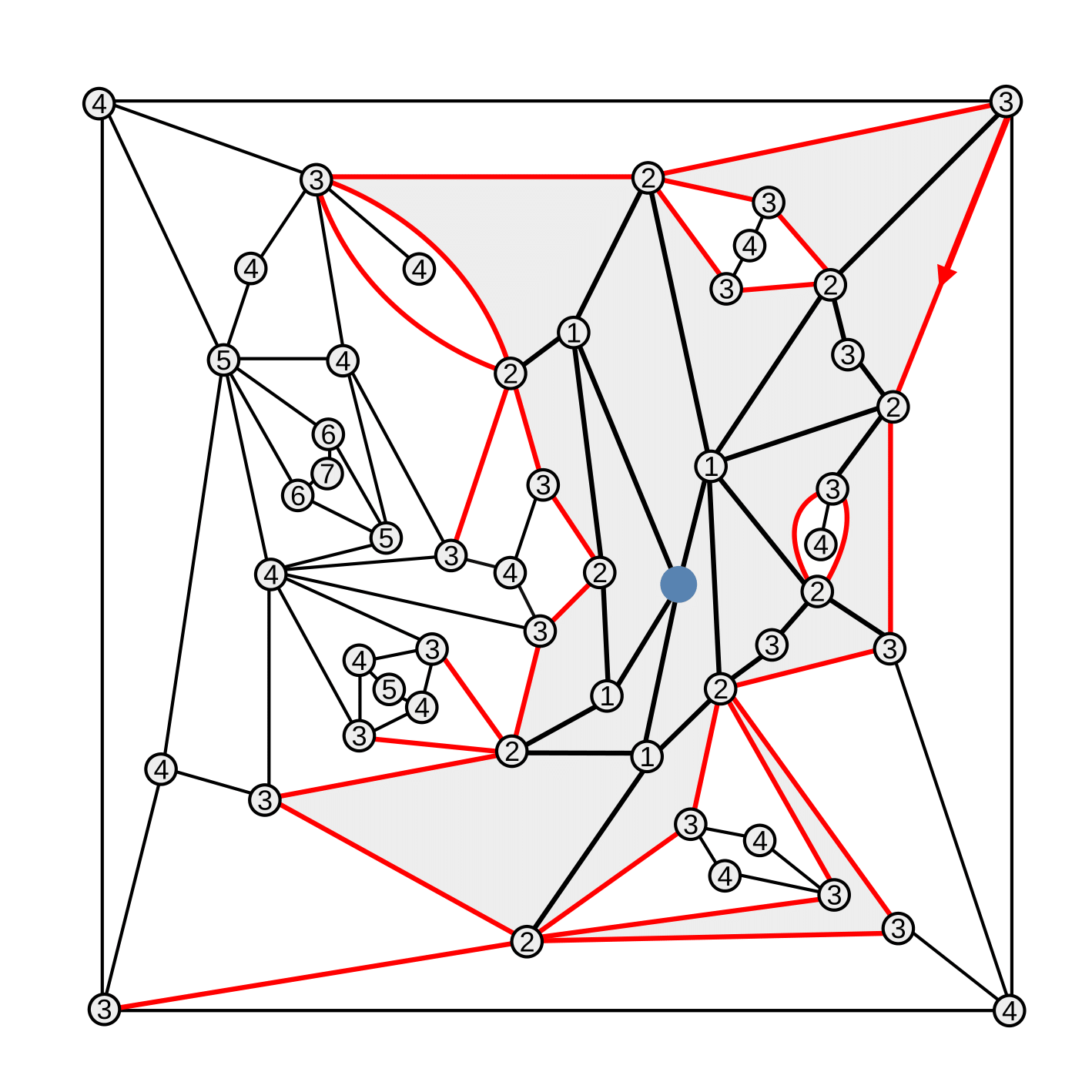}
		\end{minipage}
		\begin{minipage}[b]{0.49\textwidth}
			\includegraphics[width=\textwidth]{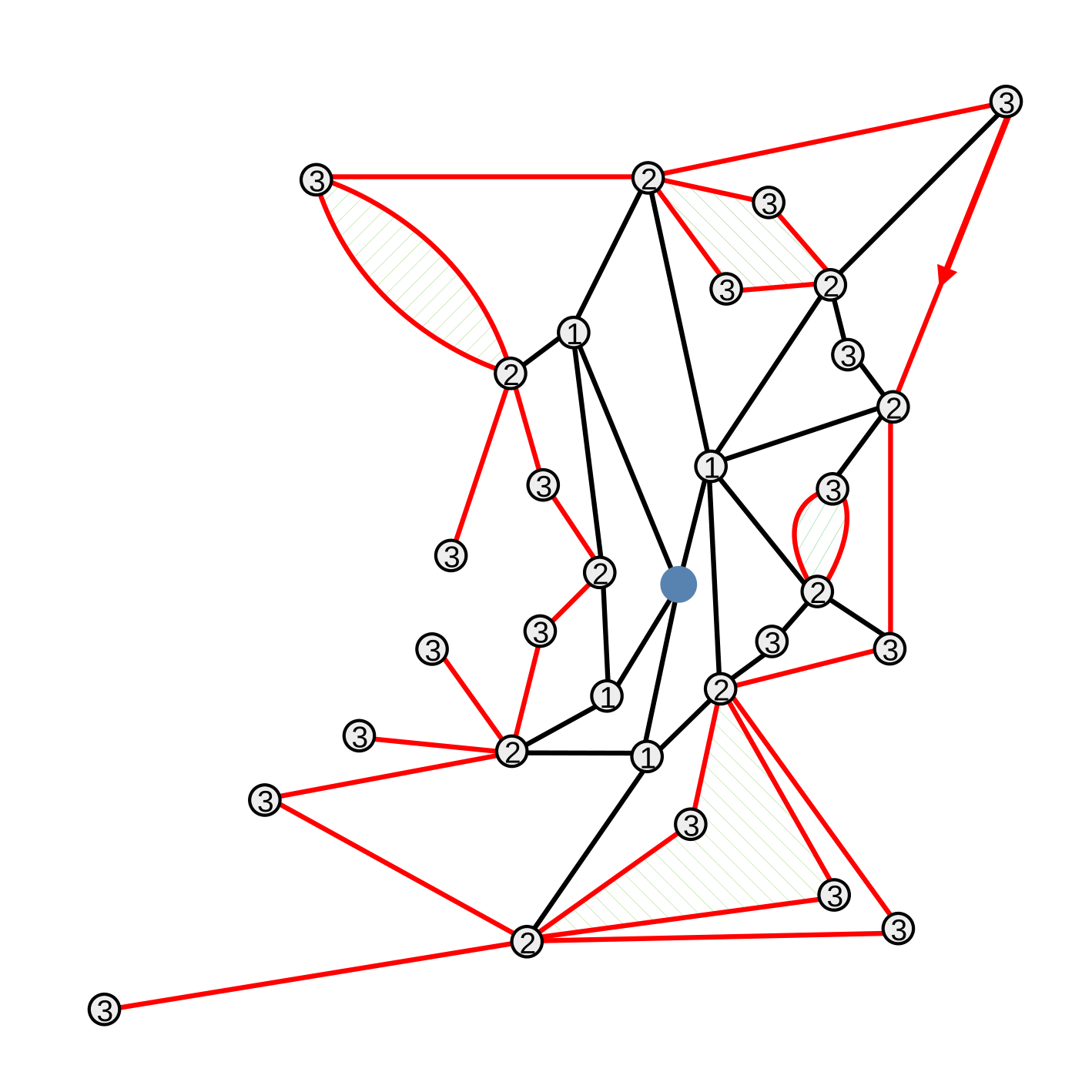}
		\end{minipage}
		
		\caption{An element $\mathbf q$ of $\mathcal Q_{50}^\bullet$, the distinguished vertex is in blue and the root edge is marked by an arrow. The numbers in the circles correspond to the distance to the distinguished point. The ball of radius $3$ is obtained by only keeping edges between vertices of label $3$ or less. In the example, this creates $5$ external faces of respective degrees $2, 2, 4, 4$ and $22$. We have in this case $C_3(\mathbf q)=5$ and $P_3(\mathbf q)=34$.}
	\end{figure}
	
	The Boltzmann measure on $\mathcal{Q}^\bullet$ is the unique probability measure $\mathbb{P}^\bullet$ assigning weight proportional to $12^{-|\mathbf{q}|}$ to each $\mathbf{q} \in \mathcal{Q}^\bullet$. In other words,
	\[
	\mathbb{P}^\bullet(\mathbf{q}) = \frac{1}{Z^\bullet} \cdot 12^{-|\mathbf{q}|},
	\]
	where $ \displaystyle
	Z^\bullet = \sum_{n \geq 1} \text{Card}(\mathcal{Q}_n^\bullet) \cdot 12^{-n}$,
	(in fact, it can be shown that $Z^\bullet = \frac{1}{2}$).
	
	\begin{proposition}
		\label{AppliRandomMaps}
		Let $\mathbf{q}$ be sampled according to $\mathbb{P}^\bullet$. Conditionally on the pair $(P_{d_\star}(\mathbf q), C_{d_\star}(\mathbf q))$, the two processes
		\[
		\left(P_{k-1+d_\star}(\mathbf{q}), C_{k-1+d_\star}(\mathbf{q})\right)_{k \geq 1} \quad \text{and} \quad \left(P_{d_\star-k}(\mathbf{q}), C_{d_\star-k}(\mathbf{q})\right)_{k \geq 1}
		\]
		are independent time-homogeneous Markov chains. Moreover,
		\[
		\left(P_{k-1+d_\star}(\mathbf{q}), C_{k-1+d_\star}(\mathbf{q})\right)_{k \geq 1}
		\overset{(d)}{=}
		\left(P_{d_\star-k}(\mathbf{q}), \tfrac{1}{2} P_{d_\star-k}(\mathbf{q}) - C_{d_\star-k}(\mathbf{q}) + 1\right)_{k \geq 1}.
		\]
	\end{proposition}
	
	\begin{proof}
		The idea is to apply Corollary~\ref{CoroMarkov} to the vertical edge profile associated with the tree derived from $\mathbf{q}$ via the so-called Schaeffer bijection.
		For our purposes, we need to recall the main features of this bijection, which relates quadrangulations to well-labelled trees. (A more detailed description can be found, e.g., in~\cite{ChassaingSchaeffer}.) Fix $n \geq 1$ and $q \in \mathcal{Q}_n^\bullet$. We will define a tree $T$ with vertex set $V(T) = V(\mathbf{q}) \setminus \{\mathbf{x}_\star\}$, where each edge of $T$ corresponds to a face of $q$.	To this end, for every vertex $v \in V(q)$, define
		\[
		\ell_v = \Delta_q(\mathbf{x}_\star, v) - d_\star,
		\]
		where $d_\star$ was defined in~\eqref{dstar}. Now consider a face of $q$ with vertices $u_1, u_2, u_3, u_4$ in clockwise order (possibly with $u_1 = u_3$ or $u_2 = u_4$). Then, up to a cyclic permutation of the $u_i$, we have two possibilities:
		\begin{enumerate}
			\item $(\ell_{u_1}, \ell_{u_2}, \ell_{u_3}, \ell_{u_4}) = (m, m-1, m, m-1)$ for some $m \in \mathbb{Z}$; in this case, we include the edge $(u_1, u_3)$ in $E(T)$.
			\item $(\ell_{u_1}, \ell_{u_2}, \ell_{u_3}, \ell_{u_4}) = (m+1, m, m-1, m)$ for some $m \in \mathbb{Z}$; in this case, we include the edge $(u_1, u_2)$ in $E(T)$.
		\end{enumerate}
		One can show that the graph $T$ with vertex set $V(T)$ and edge set $E(T)$ is a tree, which inherits a planar structure from the planar embedding of $q$. The root vertex of $T$ is the endpoint $\mathbf x_0$ of $\mathbf x_1$ of the root edge $e$ that is farther from $\mathbf x_\star$. There is a convention to choose the root edge of $T$ (connecting the root vertex to its "first" child) but we omit the details since this is not relevant in what follows (see \cite{ChassaingSchaeffer} or \cite{BettinelliThesis} for details). In particular, by definition of $d_\star$, the root vertex $\rho$ of $T$ satisfies $\ell_\rho = 0$, and the pair $(T, \ell)$ then defines an element of $\mathbb{T}_0^n$, the set of (well-)labelled plane trees rooted at $0$ with $n$ vertices.
		(We mention that this construction almost defines a bijection between elements of $\mathcal{Q}_n^\bullet$ and $\mathbb{T}_0^n$, except for the loss of the root edge orientation. Encoding this orientation as an element of $\{-1, 1\}$, yields a bijection between $\mathcal{Q}_n^\bullet$ and $\mathbb{T}_0^n \times \{-1, 1\}$.)
		
		More importantly, the key observation for our purposes is that when $\mathbf{q}$ is sampled as a Boltzmann quadrangulation, the associated labelled tree $(T, \ell_T)$ has the law of a Galton--Watson tree with geometric offspring distribution, conditioned to have at least one edge, and such that label increments along edges are independent and uniformly distributed in $\{-1, 0, 1\}$.
		We define the processes $X_m^+$, $X_m^-$, $\check{X}_m^+$ and $\check{X}_m^-$ associated with $T$ as in the previous sections and recall that by Corollary~\ref{CoroMarkov}, conditionally on $(X_1^+, X_1^-)$, the two processes
		\[
		(X_m^+, X_m^-)_{m \geq 1} \quad \text{and} \quad (\check{X}_m^-, \check{X}_m^+)_{m \geq 1}
		\]
		are independent, identically distributed, time-homogeneous Markov chains.
		
		We now explain how the quantities $P_k(\mathbf{q})$ and $C_k(\mathbf{q})$ relate to these Markov chains. Fix $k \geq 1$, and note that along the boundary of an external face of $B_k(\mathbf{q})$, labels alternate between the values $k - d_\star$ and $k - d_\star - 1$. In particular, $P_k(\mathbf{q})$ is equal to twice the number of corners with label $k - d_\star - 1$ among all external faces of $B_k(\mathbf{q})$.
		It is easy to check that these corners correspond exactly to those with label $k - d_\star - 1$ that lie in faces $f$ of $\mathbf{q}$ whose vertex labels are (up to cyclic permutations) of the form
		\[
		(k - d_\star + 1, \ k - d_\star, \ k - d_\star - 1, \ k - d_\star).
		\]
		Hence, $P_k(\mathbf{q})$ is exactly twice the number of such faces in $\mathbf{q}$.
			According to the Schaeffer bijection, these faces are in one-to-one correspondence with edges of $(T, \ell_T)$ connecting vertices labelled $k - d_\star + 1$ and $k - d_\star$. Therefore,
		\[
		\frac{1}{2} P_k(\mathbf{q}) =
		\begin{cases}
			\check{X}_{d_\star - k}^+ + \check{X}_{d_\star - k}^- & \text{if } k < d_\star, \\
			X_{k - d_\star + 1}^+ + X_{k - d_\star + 1}^- & \text{if } k \geq d_\star.
		\end{cases}
		\]
			Moreover, we claim that
		\begin{align}
			\label{Ck}
			C_k(\mathbf{q}) =
			\begin{cases}
				\check{X}_{d_\star - k}^+ + 1 & \text{if } k < d_\star, \\
				X_{k - d_\star + 1}^+ & \text{if } k \geq d_\star.
			\end{cases}
		\end{align}
		Assuming~\eqref{Ck}, we obtain:
		\begin{align*}
			&\left(P_{k - 1 + d_\star}(\mathbf{q}),\ C_{k - 1 + d_\star}(\mathbf{q})\right)_{k \geq 1}
			= \left(2(X_k^+ + X_k^-),\ X_k^+\right)_{k \geq 1}, \\
			&\left(P_{d_\star - k}(\mathbf{q}),\ C_{d_\star - k}(\mathbf{q})\right)_{k \geq 1}
			= \left(2(\check{X}_k^+ + \check{X}_k^-),\ \check{X}_k^+ + 1\right)_{k \geq 1}.
		\end{align*}
		Theorem~\ref{AppliRandomMaps} then follows directly from Corollary~\ref{CoroMarkov}.\\
		
		We now proceed to prove~\eqref{Ck}. Consider the connected components (which may be single vertices) obtained from $T$ by deleting all edges between a vertex labelled $k - d_\star$ and one labelled $k - d_\star + 1$. Consider the components above level $k - d_\star + 1$ (that is, those containing only vertices with labels greater than or equal to $k - d_\star + 1$). When $k < d_\star$, one of these components contains the root vertex of $T$, and the others are in one-to-one correspondence with the tree excursions above level $k - d_\star$ in the excursion forest decomposition introduced in Section~2.2. In particular, the number of such components is exactly $\check{X}_{d_\star - k}^+ + 1$.	When $k \geq d_\star$, the situation is simpler: the connected components are in one-to-one correspondence with the tree excursions above level $k - d_\star + 1$ in the excursion forest decomposition of Section~2.2. The number of such components is then $X_{k - d_\star + 1}^+$. In summary, proving~\eqref{Ck} amounts to showing that the connected components above label $k - d_\star + 1$ (obtained from $T$ by deleting edges between labels $k - d_\star$ and $k - d_\star + 1$) are in one-to-one correspondence with the external faces of $B_k(\mathbf{q})$.
		\begin{figure}[h]
			\centering
			\begin{minipage}[b]{0.32 \textwidth}
				\centering\includegraphics[width=0.8\textwidth]{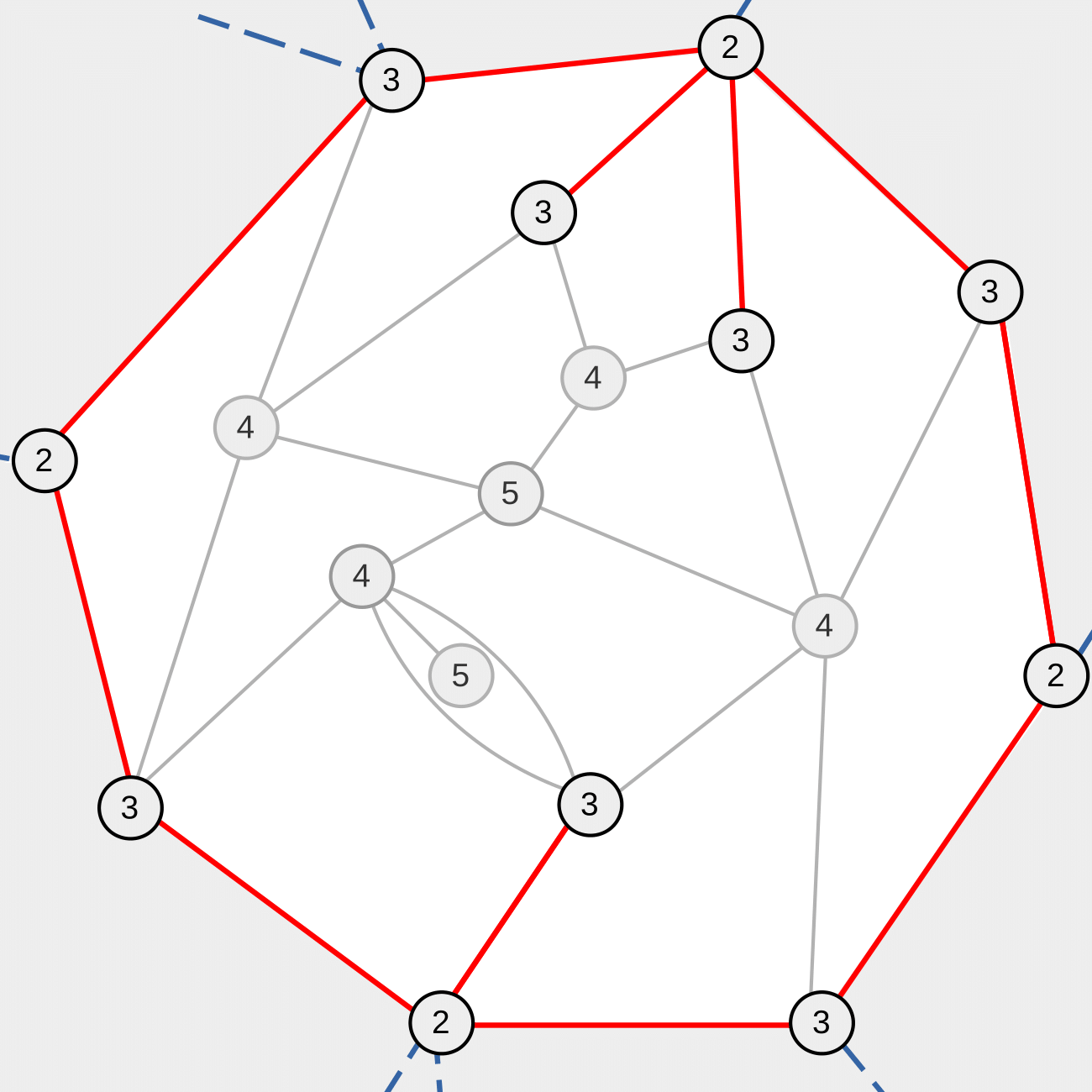}
			\end{minipage} \vrule
			\begin{minipage}[b]{0.32\textwidth}
				\centering\includegraphics[width=0.8\textwidth]{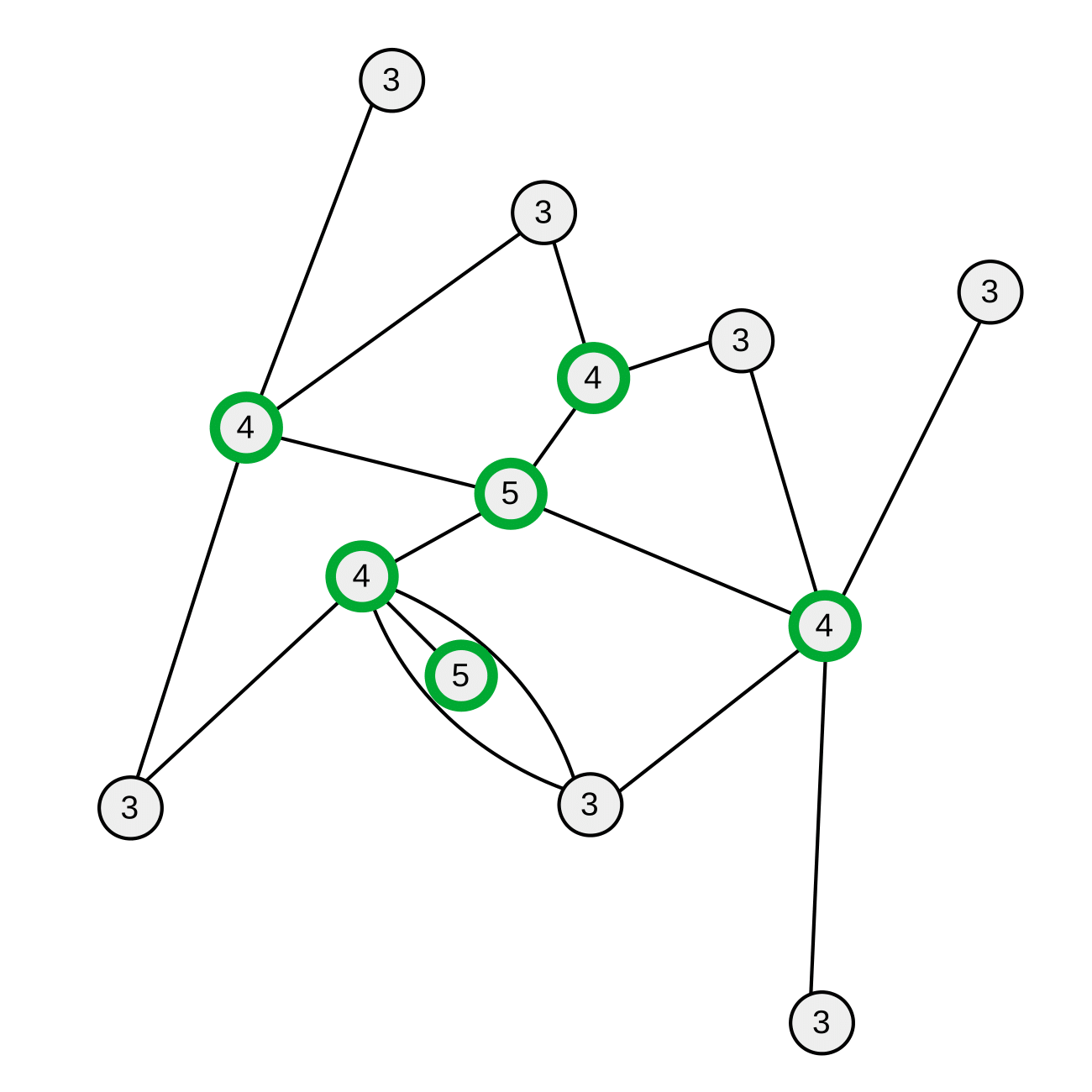}
			\end{minipage}\vrule
			\begin{minipage}[b]{0.32\textwidth}
				\centering\includegraphics[width=0.8\textwidth]{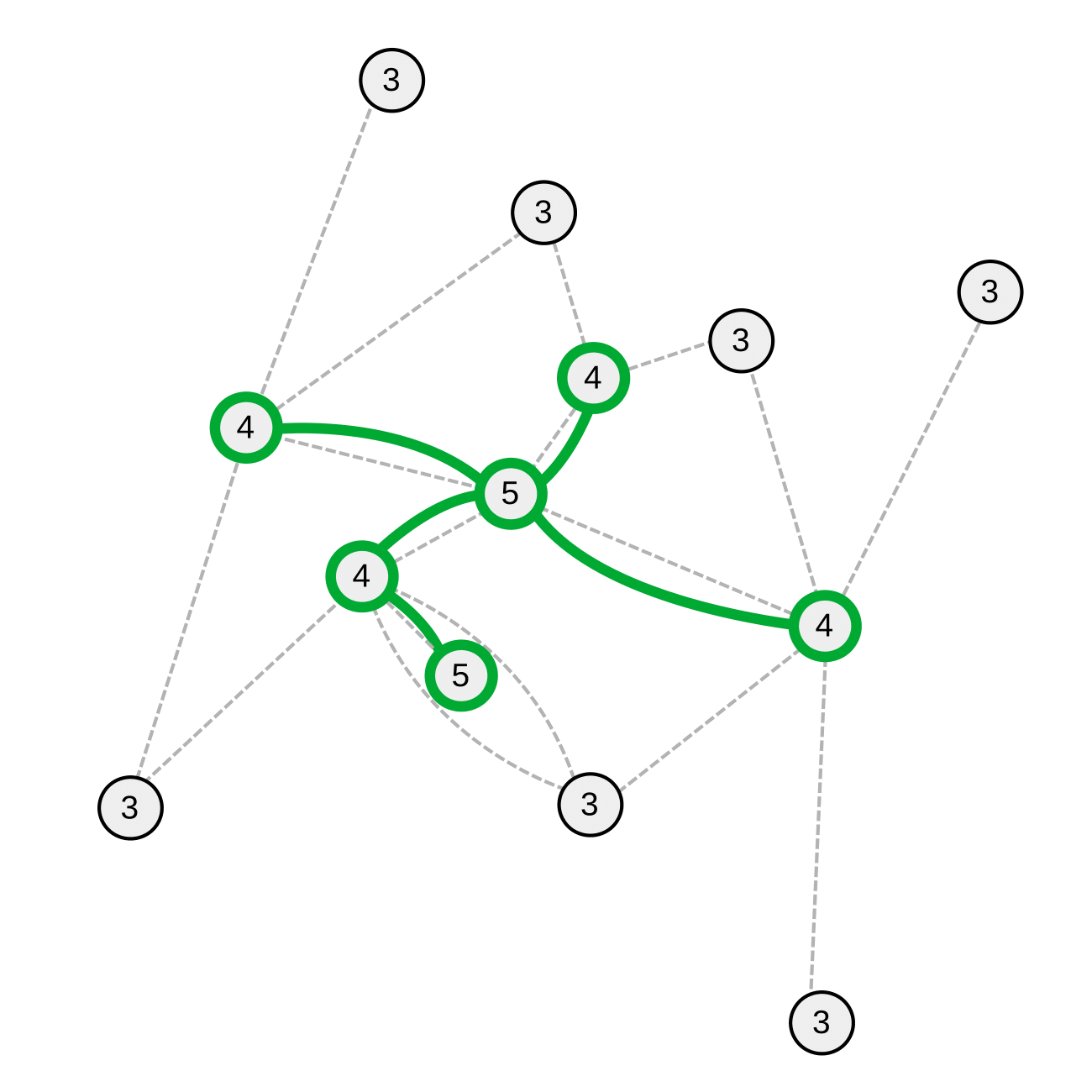}
			\end{minipage}%
			\caption{An example of a possible external face $\mathtt f$ (left) of perimeter $2p=14$ in the ball of radius $3$ (note that it may have pendant edges). Again, the numbers in the circles indicate the distance to the distinguished point $\mathbf x_\star$.  The map $\hat{\mathbf{q}}_{\mathtt f}$ (middle) is obtained by only keeping edges connecting vertices at distance at least $3$ from  $\mathbf x_\star$ that originally lie inside $\mathtt f$, it again has perimeter $14$. The number of internal faces is in this case $F=5$. The set $V_{\mathtt f}$ of vertices with label strictly greater than $k-d_\star$ is in green. The tree induced by the Schaeffer bijection (right, in green) connects all elements of $V_{\mathtt f}$.}
		\end{figure}
		
		To this end, consider an external face $\mathtt{f}$ of $B_k(\mathbf{q})$, and let $p \geq 1$ be such that $2p$ is the degree of $\mathtt{f}$. Let $\hat{\mathbf{q}}_{\mathtt{f}}$ be the submap of $\mathbf{q}$ obtained by keeping all edges that connect vertices with labels greater than or equal to $k - d_\star$ that originally lie inside $\mathtt{f}$ in $\mathbf{q}$ (see Fig. 4). It is easy to see that the map $\hat{\mathbf{q}}_{\mathtt{f}}$ also has a boundary of degree $2p$, details are left to the reader. Let $F$ be the number internal faces of degree $4$ in $\hat{\mathbf{q}}_{\mathtt{f}}$ (hence $\hat{\mathbf{q}}_{\mathtt{f}}$ has $F+1$ faces including the external face). Since each edge borders two faces, the number $E$ of edges in $\hat{\mathbf{q}}_{\mathtt{f}}$ is given by:
		\[
		E = \frac{1}{2}(4F + 2p) = 2F + p.
		\]
		By Euler's formula, the number of vertices in $\hat{\mathbf{q}}_{\mathtt{f}}$ is:
		\[
		\text{Card}(V(\hat{\mathbf{q}}_{\mathtt{f}})) =2 - (F + 1) + (2F + p) = F + 1 + p.
		\]
		Let $V_{\mathtt{f}}$ denote the subset of vertices in $\hat{\mathbf{q}}_{\mathtt{f}}$ with labels greater than or equal to $k - d_\star + 1$. In other words, $V_{\mathtt{f}}$ includes all vertices of $\hat{\mathbf{q}}_{\mathtt{f}}$ except the $p$ vertices of label $k-d_{\star}$ that were originally on the boundary of $\mathtt{f}$ and which belong to $\hat{\mathbf{q}}_{\mathtt{f}}$. Thus,
		\[
		\text{Card}(V_{\mathtt{f}}) = \text{Card}(V(\hat{\mathbf{q}}_{\mathtt{f}})) - p = F + 1.
		\]
		In particular, $V_{\mathtt{f}}$ is nonempty.

		Note that each internal face of $\hat{\mathbf{q}}_{\mathtt{f}}$ corresponds via the Schaeffer bijection to an edge in $E(T)$ connecting two vertices in $V_{\mathtt{f}}$. Hence, the subgraph of $T$ induced by $V_{\mathtt{f}}$ is a cycle-free graph with $F + 1$ vertices and at least $F$ edges. It follows that this subgraph is a tree, and therefore connected. This implies that all elements of $V_{\mathtt{f}}$ belong to one of the same connected component "above level $k-d_\star+1$" described earlier.
		
		Conversely, if $u$ and $v$ are two vertices of $\mathbf{q}$ that belong respectively to $V_{\mathtt{f}}$ and $V_{\mathtt{f'}}$ for distinct external faces $\mathtt{f}$ and $\mathtt{f'}$. It is easy to see, using the Schaeffer bijection, that any path in $T$ connecting $u$ and $v$ must pass through a vertex on the boundary of $\mathtt{f}$, and thus through a vertex with label less than or equal to $k - d_\star$. This implies that $u$ and $v$ belong to different connected components above label $k-d_\star+1$.

		The proof of~\eqref{Ck} follows from the discussion above.
	\end{proof}
	
	\begin{small}
		\printbibliography
	\end{small}
	
\end{document}